\documentclass[final,5p,times]{elsarticle}
\usepackage{multirow}
\usepackage{booktabs}
\usepackage{amssymb}
\usepackage{mathrsfs}        
\usepackage{amsthm}          
\usepackage{amsmath}         
\usepackage{epstopdf}        
\biboptions{numbers,sort&compress}                  
\usepackage[colorlinks, citecolor=blue]{hyperref}
\hypersetup{colorlinks=false}                       
\usepackage{xcolor}                                 
\definecolor{b}{rgb}{0,0,1}                         
\usepackage{colortbl}                               
\usepackage{caption2}
\usepackage{algorithmic,algorithm}                  
\usepackage{subfigure}                              

\theoremstyle{plain}
\newtheorem{theorem}{Theorem}

\theoremstyle{definition}

\newtheorem{remark}{Remark}

\journal{Neurocomputing}

\begin{document}

\begin{frontmatter}



\title{Convolutional neural networks with fractional order gradient method}
\phantomsection
\addcontentsline{toc}{title}{Convolutional neural networks with fractional order gradient method}


\author{Dian Sheng}
\author{Yiheng Wei}
\author{Yuquan Chen}
\author{Yong Wang\corref{cor1}}
\ead{yongwang@ustc.edu.cn}
\cortext[cor1]{Yong Wang is the corresponding author. }

\address{Department of Automation, University of Science and Technology of China, Hefei, 230026, China}

\begin{abstract}
This paper proposes a fractional order gradient method for the backward propagation of convolutional neural networks. To overcome the problem that fractional order gradient method cannot converge to real extreme point, a simplified fractional order gradient method is designed based on Caputo's definition. The parameters within layers are updated by the designed gradient method, but the propagations between layers still use integer order gradients, and thus the complicated derivatives of composite functions are avoided and the chain rule will be kept. By connecting every layers in series and adding loss functions, the proposed convolutional neural networks can be trained smoothly according to various tasks. {\color{b}Some practical experiments are carried out in order to demonstrate fast convergence, high accuracy and ability to escape local optimal point at last.}
\end{abstract}


\end{frontmatter}


\section{Introduction}\label{Section 1}
Machine learning technology powers many aspects of modern society: from web searches to content filtering on social networks to recommendations on e-commerce websites, and it is increasingly presented in consumer products such as cameras and smart phones \cite{LeCun:2015Nature}. {\color{b} During the development of machine learning, a variety of artificial neural networks have been created and gradually playing a more and more important role. Although it is still far from perfection, artificial neural networks are proven to be excellent enough, especially for convolutional neural networks (CNN). Research on CNN can be traced back to $1979$. Based on visual cortex, Fukushima designed the neural networks named 'neocognitron' which is regarded as the origin of CNN \cite{Fukushima:1980BC}. As a breakthrough has been made for the back propagation neural networks (BPNN) \cite{Rumelhart:1988Book}, the first CNN was invented by applying the backward propagation in training model \cite{Alexander1990RSR}. After the emergence of two-dimensional CNN, scholars proposed a series of networks such as LeNet \cite{LeCun:1989NC,LeCun:1998IEEE}, AlexNet \cite{Krizhevsky:2012imagenet}, VGG \cite{Simonyan:2015ICLP}, GoogLeNet\cite{Szegedy:2015CVPR} and ResNet \cite{He:2016CVPR}.  However, no matter how deep or large a neural network is, the key of algorithm is gradient method in backward propagation.}

As fractional order calculus is successfully applied in LMS filtering \cite{Raja:2015SP,Cheng:2017SP,Yin:2019MSSP}, systems identification \cite{Cheng:2018SP,Cui:2018ISAT}, control theories \cite{Li:2009Automatica,Lu:2010TAC,Yin:2014Automatica,Wei:2017JOTA} and so on, there arises a new trend that introduces fractional order calculus into gradient method. Professor Pu is the first one who pays attention to fractional order gradient method. He adopts fractional order derivatives to replace the integer order derivatives in traditional gradient method directly \cite{Pu:2015TNNLS}. {\color{b} According to the definitions of fractional order derivatives \cite{Monje:2010Book}, not only current information but also historical one are taken into consideration, thus it is potential for fractional order gradient to escape local optimal point. Nevertheless, such gradient method cannot ensure the convergence to real extreme point, even if the objective function is a simple quadratic function.} To remedy this congenital defect, Chen uses truncation and short memory principle to modify the fractional order gradient method \cite{Chen:2017AMC,Chen:2018ACC} {\color{b} which turns out: it is convergent to real extreme point just as integer order method do, but with faster convergent speed than integer order method.}

During the research of fractional order gradient method, some scholars have found its application to artificial neural networks at the same time. Considering that fractional order derivatives of composite functions are complicated, professor Wang only uses fractional order gradients for updating parameters so that the chain rule will be kept to calculate integer order gradients along backward propagation \cite{Wang:2017NN}. Similar method is followed but the different structure of networks is applied in \cite{Bao:2018CIN}. {\color{b} Their experiments demonstrate that fractional order gradients improve the networks performance on accuracy, and the cost of time has changed little because of relatively simpler calculation by adopting Caputo's definition.
} However, their fractional order gradient method is based on the strict definition of fractional order derivatives, which leads to the same problem as \cite{Pu:2015TNNLS}.

Even if great efforts have been made to neural networks with fractional order gradient method, it is still a novel research and far away from perfection at present. There remain some aspects to be improved.
\begin{itemize}
\item The convergence to real extreme point is necessary for gradient method.
\item The available range of fractional order can be extended to $0 < \alpha < 2$.
\item Neural networks of more complicated structure are worth researching in depth.
\item How to use the chain rule in fractional order neural networks is still a problem.
\item Loss function may be chosen as not only quadratic function but cross-entropy function.
\end{itemize}

Therefore, this paper provides conventional CNN with a novel fractional order gradient method. To the best of our knowledge, no scholar has ever investigated the CNN by fractional order gradient method. The proposed method is creative for neural networks as well as gradient method. First, based on the Caputo's definition of fractional order derivatives, a fractional order gradient method is designed and proved to converge to real extreme point. Second, the gradients in backward propagation of neural networks are divided into two categories, namely the gradients transferred between layers and the gradients for updating parameters within layers. Third, the updating gradients are replaced by fractional order one, but transferring gradients are integer order so that the chain rule could be kept using. With connecting all layers end-to-end and adding loss functions, the CNN with fractional order gradient method is achieved. {\color{b} What's more, the proposed neural networks validate that fractional order gradients perform outstanding acts of fast convergence, good accuracy and ability to escape local optimal point.}

The remainder of this article is organized as follows. Section \ref{Section 2} introduces a fractional order gradient method and provides some basic knowledge for subsequent use.  Fractional order gradient method is recommended for the fully connected layers and convolution layers in Section \ref{Section 3}, respectively. In Section \ref{Section 4},  some experiments are provided to illustrate the validity of the proposed approach. Conclusions are given in Section \ref{Section 5}.

\section {Preliminaries}\label{Section 2}
{\color{b} A general CNN is composed of convolution layers, pooling layers and fully connected layers. Fractional order gradient method is applicable for all layers except pooling layers, since there is no need of updating parameters in pooling layers. In the fully connected layers, each node is connected to all nodes of last layer, whereas the convolution layers use many convolution kernels to scan the input. In spite of different structures, their backward propagations are almost the same, which make the study on gradient method much easier. Before the introduction of fractional order gradients within backward propagations, some preliminary knowledge needs emphasizing here.}

There are some widely accepted definitions of fractional order derivative ${{\mathscr D}^{{\alpha}}}$, such as Riemann-Liouville, Caputo and Grunwald-Letnikov, but the Caputo's one is chosen for subsequent use, since its derivative of constant equals zero. The Caputo's definition is
\begin{eqnarray}\label{Eq2.1}
\textstyle
{}_{t_0}^{}{\mathscr D}_t^\alpha f\left( t \right) = \frac{1}{{\Gamma \left( {m - \alpha } \right)}}\int_{t_0}^t {\frac{{{f^{\left( m \right)}}\left( \tau  \right)}}{{{{\left( {t - \tau } \right)}^{\alpha  - m + 1}}}}{\rm{d}}\tau },
\end{eqnarray}
where $m-1< \alpha<m, m\in \mathbb{N}^+$, $\Gamma \left( \alpha  \right) = \int_0^\infty  {{x^{\alpha  - 1}}{{\rm e}^{ - x}}{\rm{d}}x} $ is the Gamma function, $t_0$ is the initial value. Alternatively, (\ref{Eq2.1}) can be rewritten as the following form
\begin{eqnarray}\label{Eq2.2}
~{}_{t_{0}}^{}{\mathscr D}_{t}^\alpha f(t) = \sum\limits_{i = m}^\infty  {\frac{{{f^{(i)}}({t_{0}})}}{{\Gamma (i + 1 - \alpha )}}{{({t} - {t_{0}})}^{i - \alpha }}}.
\end{eqnarray}

Suppose $f(x)$ to be a smooth convex function with a unique extreme point $x^*$. It is well known that each iterative step of the conventional gradient method is formulated as
\begin{eqnarray}\label{Eq2.3}
x_{K+1} = x_{K} - \mu ~f^{(1)}(x_{K}),
\end{eqnarray}
where $\mu$ is the iterative step size or learning rate, $K$ is iterative times. Similarly, the fractional order gradient method is written as
\begin{eqnarray}\label{Eq2.4}
x_{K+1} = x_K - \mu ~{}_{x_0}^{}{\mathscr D}_{x_K}^\alpha f(x).
\end{eqnarray}
If fractional order derivatives are directly applied in (\ref{Eq2.4}), the above fractional order gradient method cannot converge to the real extreme point $x^*$, but to an extreme point under definition of fractional order derivatives, such extreme point is associated with initial value and order, generally not equal to $x^*$ \cite{Pu:2015TNNLS}.

To guarantee the convergence to real extreme point, an alternative fractional order gradient method \cite{Chen:2017AMC} is considered via following iterative step
\begin{eqnarray}\label{Eq2.5}
x_{K+1} = x_{K} - \mu ~{}_{x_{K-1}}^{}{\mathscr D}_{x_{K}}^\alpha f(x),
\end{eqnarray}
with $0 < \alpha < 1$ and
\begin{eqnarray}\label{Eq2.6}
~{}_{x_{K-1}}^{}{\mathscr D}_{x_{K}}^\alpha f(x) = \sum\limits_{i = 0}^\infty  {\frac{{{f^{(i + 1)}}({x_{K-1}})}}{{\Gamma (i + 2 - \alpha )}}{{({x_{K}} - {x_{K-1}})}^{i + 1 - \alpha }}}.
\end{eqnarray}
When only the first item is reserved and its absolute value is introduced, the fractional order gradient method with $0<\alpha<2$ is simplified as
\begin{eqnarray}\label{Eq2.7}
x_{K+1} = x_{K} - \mu ~ \frac{f^{(1)}(x_{K-1})}{\Gamma(2-\alpha)}|x_{K} - x_{K-1}|^{1-\alpha}.
\end{eqnarray}

\begin{theorem}\label{Theorem1}
If fractional order gradient method (\ref{Eq2.7}) is convergent, it will converge to the real extreme point $x^*$.
\end{theorem}

\begin{proof}
It is a proof by contradiction. Assume that $x_K$ converges to a different point $X  \ne x^*$, namely $\mathop {\lim}\limits_{K \to \infty} |x_K - X| = 0$. Therefore, it can be concluded that for any sufficient small positive scalar $\varepsilon$, there exists a sufficient large number $N \in \mathbb{N}$ such that $|x_{K-1}-X|<\varepsilon<|x^*-X|$ for any $K-1>N$. Then {\color{b} $\delta  = \mathop {\inf }\limits_{K - 1 > N} |{f^{(1)}}({x_{K-1}})| > 0$ } must hold.

\noindent According to (\ref{Eq2.7}), the following inequality is obtained
\begin{eqnarray}\label{Eq2.8}
\begin{array}{rl}
|{x_{K + 1}} - {x_K}| = & \hspace{-8pt} \left| {\mu \frac{{{f^{(1)}}({x_{K - 1}})}}{{\Gamma (2 - \alpha )}}|{x_K} - {x_{K - 1}}{|^{1 - \alpha }}} \right|\\
 = & \hspace{-8pt} \mu \frac{{|{f^{(1)}}({x_{K - 1}})|}}{{\Gamma (2 - \alpha )}}|{x_K} - {x_{K - 1}}{|^{1 - \alpha }}\\
 \ge & \hspace{-8pt} d|{x_K} - {x_{K - 1}}{|^{1 - \alpha }},
\end{array}
\end{eqnarray}
with $d = \mu \frac{\delta}{\Gamma(2-\alpha)}$.

\noindent Considering that one can always find a $\varepsilon$ such that $2\varepsilon < d^{\frac{1}{\alpha}}$, then the following inequality will hold
\begin{eqnarray}\label{Eq2.9}
|x_K - x_{K-1}| \le |x_K - X| + |x_{K-1} - X| < 2\varepsilon < d^{\frac{1}{\alpha}}.
\end{eqnarray}
The above inequality could be rewritten as $d > |x_K - x_{K-1}|^\alpha$. When this inequality is introduced into (\ref{Eq2.8}), the result is
\begin{eqnarray}\label{Eq2.10}
|{x_{K + 1}} - {x_K}| > |x_K - x_{K-1}|,
\end{eqnarray}
which implies that ${x_K}$ is not convergent. It contradicts to the assumption that ${x_K}$ is convergent to $X$, thus the proof is completed.
\end{proof}

\begin{remark}\label{Remark1}
When a small positive value $\delta > 0$ is introduced, the following fractional order gradient method will avoid singularity caused by $x_K=x_{K-1}$.
\begin{eqnarray}\label{Eq2.11}
x_{K+1} = x_{K} - \mu ~ \frac{f^{(1)}(x_{K-1})}{\Gamma(2-\alpha)}(|x_{K} - x_{K-1}| + \delta)^{1-\alpha}.
\end{eqnarray}
\end{remark}

{\color{b} Compared with gradient method based on strict definition of fractional order derivatives \cite{Pu:2015TNNLS}, the modified fractional order gradient methods (\ref{Eq2.7}) and other similar methods \cite{Chen:2017AMC} are proven to be convergent to the real extreme point. To demonstrate faster convergence of proposed gradient method, a common type of objective function, namely quadratic function, is selected to show optimizing process.

The objective function is $f(x)=(x-3)^2$ where the step size is set to $0.1$ and two initial points are randomly chosen in $[0, 0.1]$. Other initializations are completely available only if the distance between two points is restricted to a small range. As is shown by Fig. \ref{Fig1}, the optimizing process is obviously promoted with the application of fractional order gradient method.}

\begin{figure}[htbp]
  \centering
  \includegraphics[width=0.9\hsize]{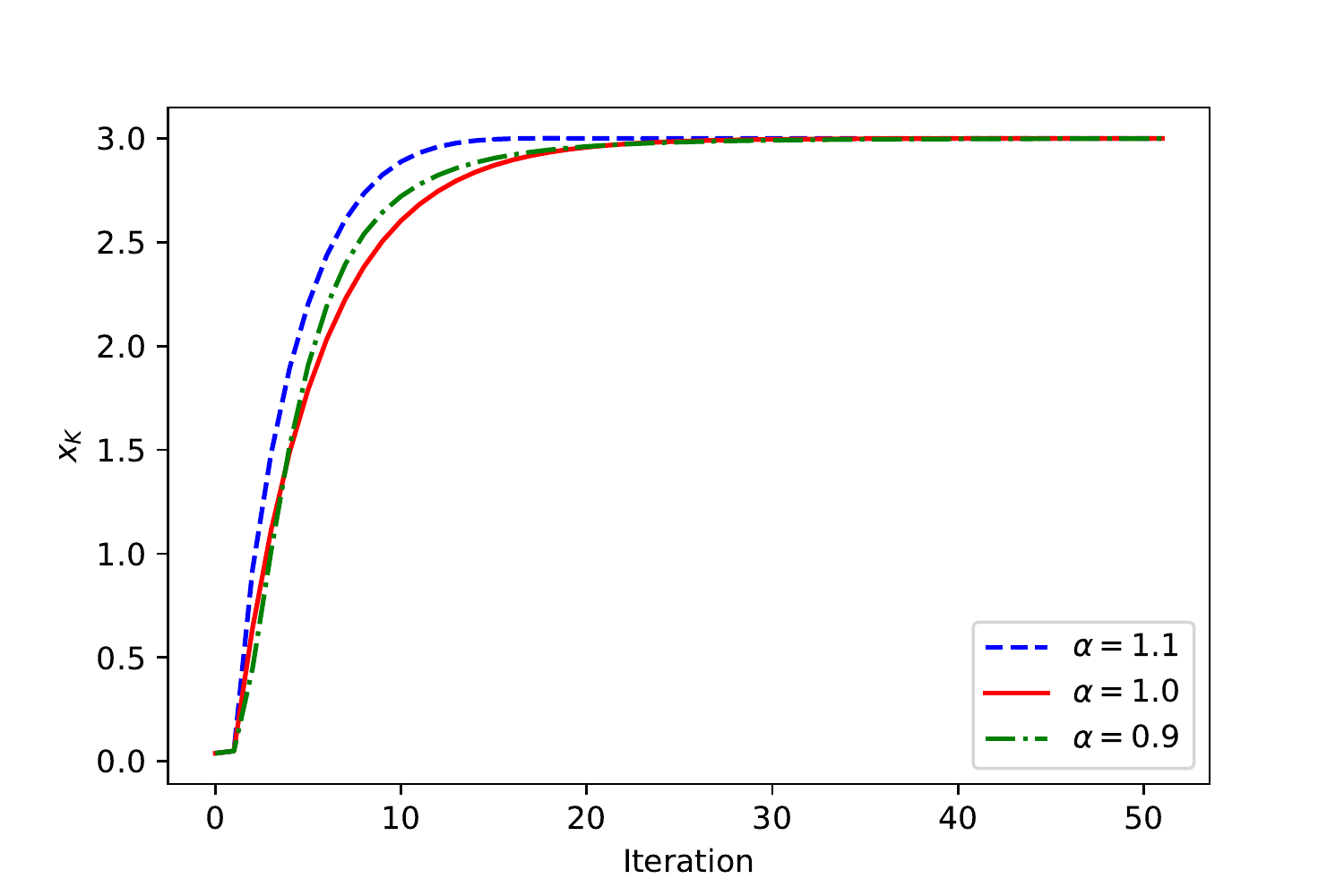}
  \caption{The performance of different gradient methods.}\label{Fig1}
\end{figure}

{\color{b} In view of theory and practice above, compared with existing integer order or fractional order gradient method, the proposed fractional order gradient method not only is convergent to real extreme point but also converges with faster speed.} 

{\color{b}\section{CNN with fractional order gradients}\label{Section 3}}
{\color{b} Although the key procedure of mathematical calculation is quite similar in convolution layers and fully connected layers, the different structures lead to different ways to research. First of all, fully connected layers with fractional order gradient are introduced.
}
\subsection{Fully connected layers}
The training procedure of neural networks contains two steps, one of them is forward propagation. Such propagation between two layers is illustrated as Fig. \ref{Fig2}, where superscript $[l]$ is the number of layer, subscript $i$ is the number of node in certain layer, $a^{[l]}_i \in \mathbb{R}$ is the output of $i$-th node in $l$-th layer.

\begin{figure}[htbp]
  \centering
  \includegraphics[width=0.45\textwidth]{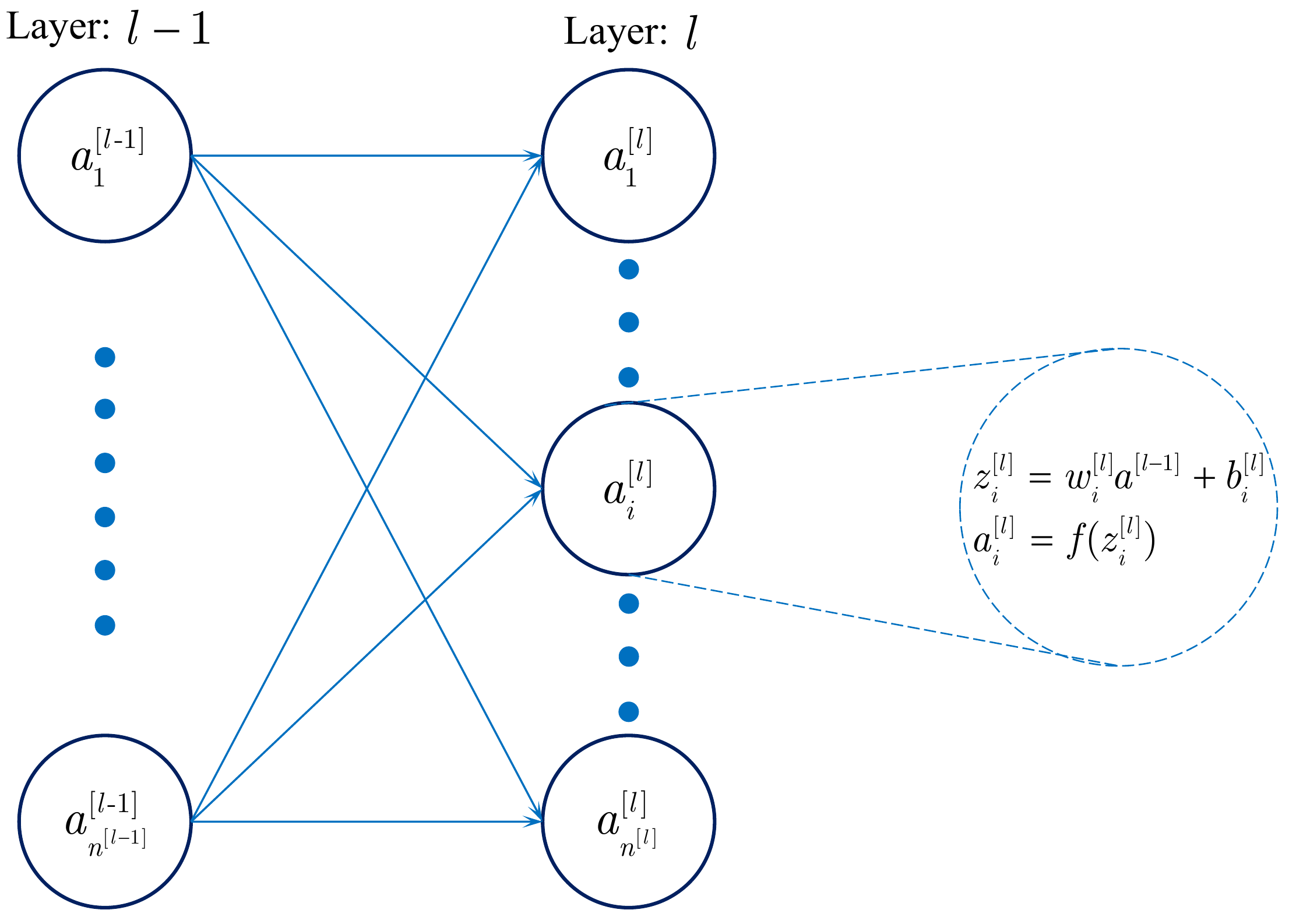}
  \caption{Forward propagation of fully connected layers.}\label{Fig2}
\end{figure}

\noindent The output $a^{[l]}_i$ is from
\begin{eqnarray}\label{Eq3.1}
\left\{ \begin{array}{rl}
z_i^{[l]} = & \hspace{-8pt} w_i^{[l]}{a^{[l - 1]}} + b_i^{[l]},\\
a_i^{[l]} = & \hspace{-8pt} f(z_i^{[l]}),
\end{array} \right.
\end{eqnarray}
where $w^{[l]}_i = [w^{[l]}_{i1}, w^{[l]}_{i2}, \cdots, w^{[l]}_{in^{[l-1]}}] \in \mathbb{R}^{n^{[l-1]}}$ is weight, $b^{[l]}_i \in \mathbb{R}$ is bias, $a^{[l-1]} = [a^{[l-1]}_{1}, a^{[l-1]}_{2}, \cdots, a^{[l-1]}_{n^{[l-1]}}]^{\rm T} \in \mathbb{R}^{n^{[l-1]}}$ is the output of last layer and function $f(\cdot)$ is activation function.

Another step of training procedure is backward propagation, in which fractional order gradient method takes the place of traditional method. Due to imperfect use of chain rule in fractional order derivatives, the gradients of backward propagation are a blend of fractional order and integer order. As is shown in Fig \ref{Fig3}, there are two types of gradients that pass through layers. One is the transferring gradient (solid line) which links nodes between two layers, the other is updating gradient (dotted line) which is used for parameters within layers. $L$ is the loss function, $\alpha$ is the fractional order, $\frac{\partial {}^\alpha L}{\partial w_i^{[l]} {}^\alpha}$ and $\frac{\partial {}^\alpha L}{\partial b_i^{[l]} {}^\alpha}$ are defined as fractional order gradients of $w_i^{[l]}$ and $b_i^{[l]}$, respectively.

\begin{figure}[htbp]
  \centering
  \includegraphics[width=0.48\textwidth]{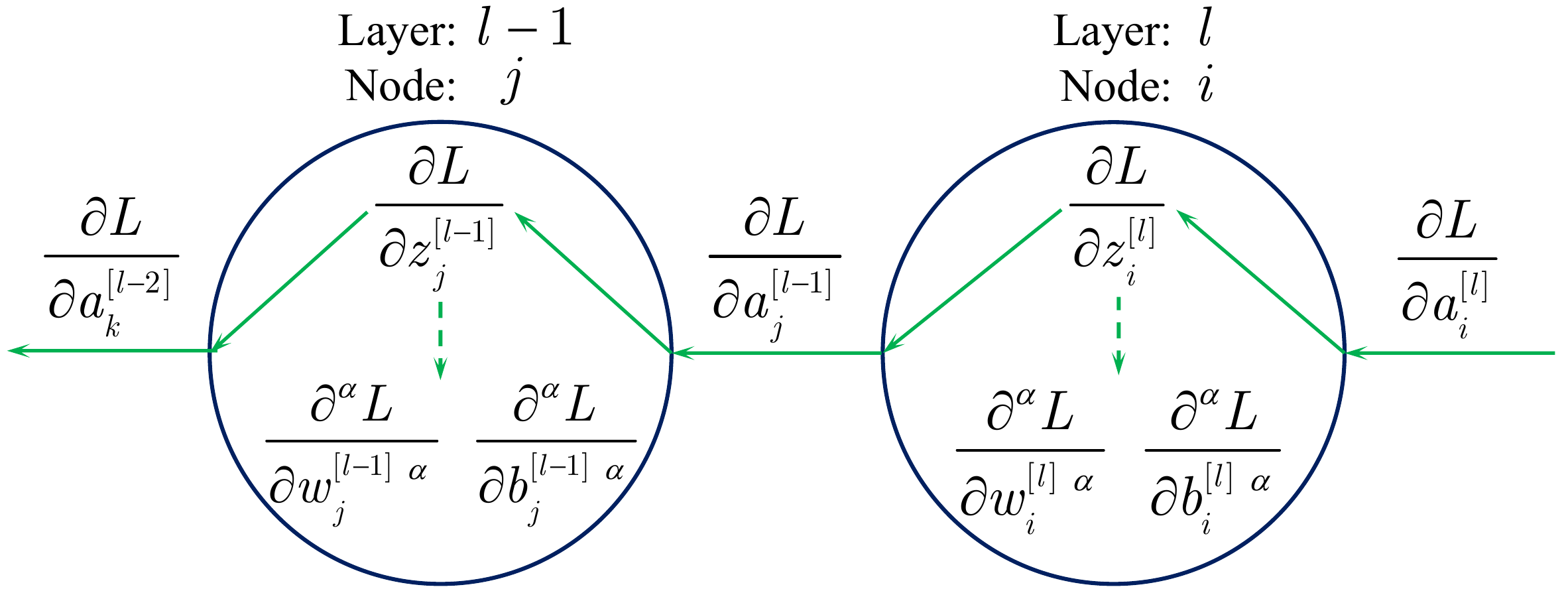}
  \caption{Backward propagation of fully connected layers.}\label{Fig3}
\end{figure}

\noindent In order to use the chain rule continuously, the transferring gradient is provided with integer order
{\color{b}
\begin{eqnarray}\label{Eq3.2}
\left\{ \begin{array}{rl}
\frac{\partial L}{\partial z_i^{[l]}} = & \hspace{-8pt} \frac{\partial L}{\partial a_i^{[l]}} \frac{\partial a_i^{[l]}}{\partial z_i^{[l]}} = \frac{\partial L}{\partial a_i^{[l]}} f^{(1)}(z_i^{[l]}),\\
\frac{\partial L}{\partial a_j^{[l-1]}} = & \hspace{-8pt} \sum\limits_{i = 1}^{n^{[l]}} \frac{\partial L}{\partial a_i^{[l]}} \frac{\partial a_i^{[l]}}{\partial z_i^{[l]}} \frac{\partial z_i^{[l]}}{\partial a_j^{[l-1]}} = \sum\limits_{i = 1}^{n^{[l]}} \frac{\partial L}{\partial z_i^{[l]}} w^{[l]}_{ij},
\end{array} \right.
\end{eqnarray}
}
but the updating gradient is replaced by fractional order
\begin{eqnarray}\label{Eq3.3}
\left\{ \begin{array}{rl}
\frac{\partial {}^\alpha L}{\partial w_{ij}^{[l]} {}^\alpha} = & \hspace{-8pt} \frac{\partial L}{\partial z_i^{[l]}} \frac{\partial {}^\alpha z_i^{[l]}}{\partial w_{ij}^{[l]} {}^\alpha} = \frac{\partial L}{\partial a_i^{[l]}} f^{(1)}(z_i^{[l]}) \frac{\partial {}^\alpha z_i^{[l]}}{\partial w_{ij}^{[l]} {}^\alpha},\\
\frac{\partial {}^\alpha L}{\partial b_i^{[l]} {}^\alpha} = & \hspace{-8pt} \frac{\partial L}{\partial z_i^{[l]}} \frac{\partial {}^\alpha z_i^{[l]}}{\partial b_i^{[l]} {}^\alpha} = \frac{\partial L}{\partial a_i^{[l]}} f^{(1)}(z_i^{[l]}) \frac{\partial {}^\alpha z_i^{[l]}}{\partial b_i^{[l]} {}^\alpha},
\end{array} \right.
\end{eqnarray}
with $i = 1, 2, \cdots, n^{[l]}$ and $j = 1, 2, \cdots, n^{[l-1]}$. When the fractional order gradient (\ref{Eq2.7}) is adopted, the gradient of the $K$-th iteration becomes
\begin{eqnarray}\label{Eq3.4}
\left\{ \begin{array}{rl}
\frac{\partial {}^\alpha z_i^{[l]}}{\partial w_{ij}^{[l]} {}^\alpha} = & \hspace{-8pt} \frac{a^{[l-1]}_{j~{(K-1)}}}{\Gamma(2-\alpha)}|w_{ij~{(K)}}^{[l]} - w_{ij~{(K-1)}}^{[l]}|^{1 - \alpha},\\
\frac{\partial {}^\alpha z_i^{[l]}}{\partial b_{i}^{[l]} {}^\alpha} = & \hspace{-8pt} \frac{1}{\Gamma(2-\alpha)}|b_{i~(K)}^{[l]} - b_{i~{(K-1)}}^{[l]}|^{1 - \alpha},
\end{array} \right.
\end{eqnarray}
where $w_{ij~(K)}^{[l]}$ and $b_{i~(K)}^{[l]}$ are parameters $w_{ij}^{[l]}$ and $b_{i}^{[l]}$ at the $K$-th iteration, $a^{[l-1]}_{j~{(K-1)}}$ is output $a^{[l-1]}_{j}$ at the ($K-1$)-th iteration. Consequently, the fractional order updating gradient is achieved by introducing (\ref{Eq3.4}) to (\ref{Eq3.3})
\begin{eqnarray}\label{Eq3.5}
\left\{ \begin{array}{rl}
\frac{\partial {}^\alpha L}{\partial w_{ij}^{[l]} {}^\alpha} =& \hspace{-8pt} \frac{\partial L}{\partial a^{[l]}_{i~{(K-1)}}} f^{(1)}(z^{[l]}_{i~{(K-1)}}) \frac{a^{[l-1]}_{j~{(K-1)}}}{\Gamma(2-\alpha)}|w_{ij~(K)}^{[l]} \\
& \hspace{-8pt} - w_{ij~{(K-1)}}^{[l]}|^{1 - \alpha},\\
\frac{\partial {}^\alpha L}{\partial b_i^{[l]} {}^\alpha} =& \hspace{-8pt} \frac{\partial L}{\partial a^{[l]}_{i~{(K-1)}}} f^{(1)}(z^{[l]}_{i~{(K-1)}}) \frac{1}{\Gamma(2-\alpha)}|b_{i~(K)}^{[l]} \\
& \hspace{-8pt} - b_{i~{(K-1)}}^{[l]}|^{1 - \alpha},
\end{array} \right.
\end{eqnarray}
where $z^{[l]}_{i~{(K-1)}}$ and $\frac{\partial L}{\partial a^{[l]}_{i~{(K-1)}}}$ are $z^{[l]}_{i}$ and $\frac{\partial L}{\partial a^{[l]}_{i}}$ at the $(K-1)$-th iteration, respectively.

Actually, samples are not input one by one in most case. When a batch of samples are input each time, (\ref{Eq3.5}) turns into
\begin{eqnarray}\label{Eq3.6}
\left\{ \begin{array}{rl}
\frac{\partial {}^\alpha L}{\partial w_{ij}^{[l]} {}^\alpha} = & \hspace{-8pt} \sum\limits_{s = 1}^{m} \frac{\partial L}{\partial z_{is}^{[l]}} \frac{\partial {}^\alpha z_{is}^{[l]}}{\partial w_{ij}^{[l]} {}^\alpha} \\
= & \hspace{-8pt} \sum\limits_{s = 1}^{m} \frac{\partial L}{\partial a^{[l]}_{is~{(K-1)}}} f^{(1)}(z^{[l]}_{is~{(K-1)}}) \frac{a^{[l-1]}_{js~{(K-1)}}}{\Gamma(2-\alpha)}|w_{ij~(K)}^{[l]} \\
& \hspace{-8pt} - w_{ij~{(K-1)}}^{[l]}|^{1 - \alpha},\\
\frac{\partial {}^\alpha L}{\partial b_i^{[l]} {}^\alpha} = & \hspace{-8pt} \sum\limits_{s = 1}^{m} \frac{\partial L}{\partial z_{is}^{[l]}} \frac{\partial {}^\alpha z_{is}^{[l]}}{\partial b_{i}^{[l]} {}^\alpha} \\
= & \hspace{-8pt} \sum\limits_{s = 1}^{m} \frac{\partial L}{\partial a^{[l]}_{is~{(K-1)}}} f^{(1)}(z^{[l]}_{is~{(K-1)}}) \frac{1}{\Gamma(2-\alpha)}|b_{i~(K)}^{[l]} \\
& \hspace{-8pt} - b_{i~{(K-1)}}^{[l]}|^{1 - \alpha},
\end{array} \right.
\end{eqnarray}
where $m$ is batch size, the subscript $s$ means $s$-th sample of a batch. After vectorization, above equations are simplified as
{\color{b}
\begin{eqnarray}\label{Eq3.7}
\left\{ \begin{array}{rl}
\frac{\partial {}^\alpha L}{\partial W^{[l]} {}^\alpha} = & \hspace{-8pt} \frac{1}{\Gamma(2-\alpha)} \left[\frac{\partial L}{\partial A^{[l]}_{~(K-1)}} \circ f^{(1)}(Z^{[l]}_{~(K-1)})\right] A^{[l-1] ~\rm T}_{~(K-1)} \\
& \hspace{-8pt} \circ |W_{~(K)}^{[l]} - W_{~{(K-1)}}^{[l]}|^{1 - \alpha}, \\
\frac{\partial {}^\alpha L}{\partial b^{[l]} {}^\alpha} = & \hspace{-8pt} \frac{1}{\Gamma(2-\alpha)} {\rm sum}\left(\frac{\partial L}{\partial A^{[l]}_{~(K-1)}} \circ f^{(1)}(Z^{[l]}_{~(K-1)})\right) \\
& \hspace{-8pt} \circ |b_{~(K)}^{[l]} - b_{~{(K-1)}}^{[l]}|^{1 - \alpha},\\
\frac{\partial L}{\partial A^{[l-1]}_{~(K-1)}} = & \hspace{-8pt} W_{~{(K-1)}}^{[l]~{\rm T}} \left[\frac{\partial L}{\partial A^{[l]}_{~(K-1)}} \circ f^{(1)}(Z^{[l]}_{~(K-1)})\right]
\end{array} \right.
\end{eqnarray}
}
where
\[
\begin{array}{*{20}{c}}
\frac{{{\partial ^\alpha }L}}{{\partial {W^{[l]\alpha }}}} = \left[ {\begin{array}{*{20}{c}}
{\frac{{{\partial ^\alpha }L}}{{\partial w_{11}^{[l]\alpha }}}}& \cdots &{\frac{{{\partial ^\alpha }L}}{{\partial w_{1{n^{[l - 1]}}}^{[l]\alpha }}}}\\
 \vdots & \ddots & \vdots \\
{\frac{{{\partial ^\alpha }L}}{{\partial w_{{n^{[l]}}1}^{[l]\alpha }}}}& \cdots &{\frac{{{\partial ^\alpha }L}}{{\partial w_{{n^{[l]}}{n^{[l - 1]}}}^{[l]\alpha }}}}
\end{array}} \right],
\frac{{{\partial ^\alpha }L}}{{\partial {b^{[l]\alpha }}}} = \left[ {\begin{array}{*{20}{c}}
{\frac{{{\partial ^\alpha }L}}{{\partial b_1^{[l]\alpha }}}}\\
 \vdots \\
{\frac{{{\partial ^\alpha }L}}{{\partial b_{{n^{[l]}}}^{[l]\alpha }}}}
\end{array}} \right],
\end{array}
\]

\[
\begin{array}{*{20}{c}}
W_{~(K)}^{[l]} = \left[ {\begin{array}{*{20}{c}}
{w_{11~(K)}^{[l]}}& \cdots &{w_{1{n^{[l - 1]}}~(K)}^{[l]}}\\
 \vdots & \ddots & \vdots \\
{w_{{n^{[l]}}1~(K)}^{[l]}}& \cdots &{w_{{n^{[l]}}{n^{[l - 1]}}~(K)}^{[l]}}
\end{array}} \right],
b_{~(K)}^{[l]} = \left[ {\begin{array}{*{20}{c}}
{b_{1~(K)}^{[l]}}\\
 \vdots \\
{b_{{n^{[l]}}~(K)}^{[l]}}
\end{array}} \right],
\end{array}
\]

\[
A_{~(K - 1)}^{[l - 1]{\rm T}} = {\left[ {\begin{array}{*{20}{c}}
{a_{11~(K - 1)}^{[l - 1]}}& \cdots &{a_{1m~(K - 1)}^{[l - 1]}}\\
 \vdots & \ddots & \vdots \\
{a_{{n^{[l - 1]}}1~(K - 1)}^{[l - 1]}}& \cdots &{a_{{n^{[l - 1]}}m~(K - 1)}^{[l - 1]}}
\end{array}} \right]^{\rm T}},
\]

\[
{\frac{{\partial L}}{{\partial A_{~(K - 1)}^{[l]}}} = \left[ {\begin{array}{*{20}{c}}
{\frac{{\partial L}}{{\partial a_{11~(K - 1)}^{[l]}}}}& \cdots &{\frac{{\partial L}}{{\partial a_{1m~(K - 1)}^{[l]}}}}\\
 \vdots & \ddots & \vdots \\
{\frac{{\partial L}}{{\partial a_{{n^{[l]}}1~(K - 1)}^{[l]}}}}& \cdots &{\frac{{\partial L}}{{\partial a_{{n^{[l]}}m~(K - 1)}^{[l]}}}}
\end{array}} \right],}
\]

\[
{{f^{(1)}}(Z_{~(K - 1)}^{[l]}) = \left[ {\begin{array}{*{20}{c}}
{{f^{(1)}}(z_{11~(K - 1)}^{[l]})}& \cdots &{{f^{(1)}}(z_{1m~(K - 1)}^{[l]})}\\
 \vdots & \ddots & \vdots \\
{{f^{(1)}}(z_{{n^{[l]}}1~(K - 1)}^{[l]})}& \cdots &{{f^{(1)}}(z_{{n^{[l]}}m~(K - 1)}^{[l]})}
\end{array}} \right],}
\]

\noindent signs like $-, |\cdot|$ and $(\cdot)^{1-\alpha}$ are the element-wise calculation, {\color{b}$\circ$ is the Hadamard product}, $\rm sum(\cdot)$ is the sum of a matrix along horizontal axis. Then the updating parameters of fully connected layers can be summarized as
\begin{eqnarray}\label{Eq3.8}
\left\{ \begin{array}{rl}
W^{[l]}_{~(K+1)} = & \hspace{-8pt} W^{[l]}_{~(K)} - \mu ~\frac{\partial {}^\alpha L}{\partial W^{[l]} {}^\alpha},\\
b^{[l]}_{~(K+1)} = & \hspace{-8pt} b^{[l]}_{~(K)} - \mu  ~\frac{\partial {}^\alpha L}{\partial b^{[l]} {}^\alpha}.
\end{array} \right.
\end{eqnarray}

\begin{theorem}\label{Theorem2}
The fully connected layers updated by fractional order gradient method (\ref{Eq3.7}, \ref{Eq3.8}) are convergent to real extreme point.
\end{theorem}

{\color{b}
\begin{proof}
When integer order gradients are adopted in backward propagation, the gradients of $W_{(K-1)}^{[l]}$ and $b_{(K-1)}^{[l]}$ can be written as
\begin{eqnarray}\label{Eq3.81}
\left\{ \begin{array}{rl}
\frac{\partial L}{\partial W^{[l]}_{(K-1)}} = & \hspace{-8pt} \left[\frac{\partial L}{\partial A^{[l]}_{~(K-1)}} \circ f^{(1)}(Z^{[l]}_{~(K-1)})\right] A^{[l-1] ~\rm T}_{~(K-1)}, \\
\frac{\partial L}{\partial b^{[l]}_{(K-1)}} = & \hspace{-8pt} {\rm sum}\left(\frac{\partial L}{\partial A^{[l]}_{~(K-1)}} \circ f^{(1)}(Z^{[l]}_{~(K-1)})\right). \\
\end{array} \right.
\end{eqnarray}

\noindent Thus fractional order gradients (\ref{Eq3.7}) turn into
\begin{eqnarray}\label{Eq3.82}
\left\{ \begin{array}{rl}
\frac{\partial {}^\alpha L}{\partial W^{[l]} {}^\alpha} = & \hspace{-8pt} \frac{1}{\Gamma(2-\alpha)} \frac{\partial L}{\partial W^{[l]}_{(K-1)}} \circ |W_{~(K)}^{[l]} - W_{~{(K-1)}}^{[l]}|^{1 - \alpha}, \\
\frac{\partial {}^\alpha L}{\partial b^{[l]} {}^\alpha} = & \hspace{-8pt} \frac{1}{\Gamma(2-\alpha)} \frac{\partial L}{\partial b^{[l]}_{(K-1)}} \circ |b_{~(K)}^{[l]} - b_{~{(K-1)}}^{[l]}|^{1 - \alpha}.\\
\end{array} \right.
\end{eqnarray}

\noindent Then the updating parameters (\ref{Eq3.8}) becomes
\begin{eqnarray}\label{Eq3.83}
\left\{ \begin{array}{rl}
W^{[l]}_{~(K+1)} = & \hspace{-8pt} W^{[l]}_{~(K)} - \mu ~\frac{1}{\Gamma(2-\alpha)} \frac{\partial L}{\partial W^{[l]}_{(K-1)}} \\
& \hspace{-8pt} \circ |W_{~(K)}^{[l]} - W_{~{(K-1)}}^{[l]}|^{1 - \alpha},\\
b^{[l]}_{~(K+1)} = & \hspace{-8pt} b^{[l]}_{~(K)} - \mu  ~\frac{1}{\Gamma(2-\alpha)} \frac{\partial L}{\partial b^{[l]}_{(K-1)}} \\
& \hspace{-8pt} \circ |b_{~(K)}^{[l]} - b_{~{(K-1)}}^{[l]}|^{1 - \alpha}.
\end{array} \right.
\end{eqnarray}

\noindent For a certain element in $W^{[l]}$ or $b^{[l]}$, the $w^{[l]}_{ij}$ or $b^{[l]}_{i}$ can be updated by
\begin{eqnarray}\label{Eq3.84}
\left\{ \begin{array}{rl}
w^{[l]}_{ij~(K+1)} = & \hspace{-8pt} w^{[l]}_{ij~(K)} - \mu ~\frac{1}{\Gamma(2-\alpha)} \frac{\partial L}{\partial w^{[l]}_{ij~(K-1)}} \\
& \hspace{-8pt} |w_{ij~(K)}^{[l]} - w_{ij~{(K-1)}}^{[l]}|^{1 - \alpha},\\
b^{[l]}_{i~(K+1)} = & \hspace{-8pt} b^{[l]}_{i~(K)} - \mu  ~\frac{1}{\Gamma(2-\alpha)} \frac{\partial L}{\partial b^{[l]}_{i~(K-1)}} \\
& \hspace{-8pt} |b_{i~(K)}^{[l]} - b_{i~{(K-1)}}^{[l]}|^{1 - \alpha}.
\end{array} \right.
\end{eqnarray}

It is similar to the proof of Theorem \ref{Theorem1}. Assume that $w^{[l]}_{ij~(K)}$ converges to a point $w^{[l]\prime}_{ij}$ that is different from real extreme point $w^{[l]*}_{ij}$, namely $\mathop {\lim}\limits_{K \to \infty} |w^{[l]}_{ij~(K)} - w^{[l]\prime}_{ij}| = 0$. Therefore, it can be concluded that for any sufficient small positive scalar $\varepsilon$, there exists a sufficient large number $N \in \mathbb{N}$ such that $|w^{[l]}_{ij~(K)} - w^{[l]\prime}_{ij}| < \varepsilon < |w^{[l]*}_{ij}-w^{[l]\prime}_{ij}|$ for any $K-1>N$. Then {\color{b} $\delta  = \mathop {\inf }\limits_{K - 1 > N} |\frac{\partial L}{\partial w^{[l]}_{ij~(K-1)}}| > 0$ } must hold.

\noindent According to (\ref{Eq3.84}), the following inequality is obtained
\begin{eqnarray}\label{Eq3.85}
\begin{array}{l}
| w^{[l]}_{ij~(K+1)} - w^{[l]}_{ij~(K)} | \\
\hspace{-9pt} =  \left| \mu ~\frac{1}{\Gamma(2-\alpha)} \frac{\partial L}{\partial w^{[l]}_{ij~(K-1)}}  |w_{ij~(K)}^{[l]} - w_{ij~{(K-1)}}^{[l]}|^{1 - \alpha} \right|\\
\hspace{-9pt} =   \frac{\mu}{\Gamma(2-\alpha)}  \left| \frac{\partial L}{\partial w^{[l]}_{ij~(K-1)}} \right| |w_{ij~(K)}^{[l]} - w_{ij~{(K-1)}}^{[l]}|^{1 - \alpha} \\
\hspace{-9pt} \ge  d |w_{ij~(K)}^{[l]} - w_{ij~{(K-1)}}^{[l]}|^{1 - \alpha},
\end{array}
\end{eqnarray}
with $d = \frac{\mu \delta}{\Gamma(2-\alpha)}$.

\noindent Considering that one can always find a $\varepsilon$ such that $2\varepsilon < d^{\frac{1}{\alpha}}$, then the following inequality will hold
\begin{eqnarray}\label{E3.86}
\begin{array}{l}
|w_{ij~(K)}^{[l]} - w_{ij~{(K-1)}}^{[l]}| \\
\hspace{-9pt} \le |w_{ij~(K)}^{[l]} - w^{[l]\prime}_{ij}| + |w_{ij~{(K-1)}}^{[l]} - w^{[l]\prime}_{ij}| \\
\hspace{-9pt} < 2\varepsilon < d^{\frac{1}{\alpha}}.
\end{array}
\end{eqnarray}
The above inequality could be rewritten as $d > |w_{ij~(K)}^{[l]} - w_{ij~{(K-1)}}^{[l]}|^\alpha$. When this inequality is introduced into (\ref{Eq3.85}), the result is
\begin{eqnarray}\label{Eq3.87}
| w^{[l]}_{ij~(K+1)} - w^{[l]}_{ij~(K)} | > |w_{ij~(K)}^{[l]} - w_{ij~{(K-1)}}^{[l]}|,
\end{eqnarray}
which implies that $w^{[l]}_{ij~(K)}$ is not convergent. Similarly, the same result will be easily obtained for  $b^{[l]}_{i~(K)}$. It contradicts to the assumption mentioned before, thus the proof is completed
\end{proof}
}

Compared with integer order backward propagation \cite{Rumelhart:1988Nature}, the same transferring gradient $\frac{\partial L}{\partial A^{[l-1]}_{~(K-1)}}$ is kept, but the difference exists in updating gradient where the order is changed as $\frac{\partial {}^\alpha L}{\partial W^{[l]} {}^\alpha}$ and $\frac{\partial {}^\alpha L}{\partial b^{[l]} {}^\alpha}$. {\color{b} Even so, based on the Theorem \ref{Theorem2}, $W^{[l]}$ and $b^{[l]}$ updated by fractional order gradients will converge to the same real extreme points as integer order gradients do.}

\begin{remark}\label{Remark2}
Because of integer order transferring gradient, the chain rule is still available for the proposed gradient method (\ref{Eq3.7}, \ref{Eq3.8}), which avoids complicated calculation caused by fractional order derivatives, especially derivatives of activation function. As modified fractional order gradient (\ref{Eq2.7}) is applied smoothly, the speed of convergence is improved and real extreme point can be reached now.
\end{remark}

\subsection{Convolution Layers}
Although the key calculation of convolution layers is similar to fully connected layers, the complicated structure makes its iterative algorithm different. It is hard to understand the algorithm without help of figures or auxiliary descriptions.

For subsequent research, the forward propagation of convolution layers is drawn briefly in Fig. \ref{Fig4} where $a^{[l]} \in \mathbb R^{n_H^{[l]} \times n_W^{[l]} \times n_C^{[l]}}$ is the output of the $l$-th layer, $w_c^{[l]} \in \mathbb R^{F^{[l]} \times F^{[l]} \times n_C^{[l-1]}}$ and $b_c^{[l]} \in \mathbb R$ are weight and bias for channel $c$, $F^{[l]}$ is the size of convolution kernel, ${a'}^{[l-1]}$ is a slice of $a^{[l-1]}$ by selecting $F^{[l]}$ rows and $F^{[l]}$ columns over all channels (red cube), $n_H^{[l]}$, $n_W^{[l]}$ and $n_C^{[l]}$ are height, width and channels of output $a^{[l]}$, respectively.

\begin{figure}[htbp]
  \centering
  \includegraphics[width=0.48\textwidth]{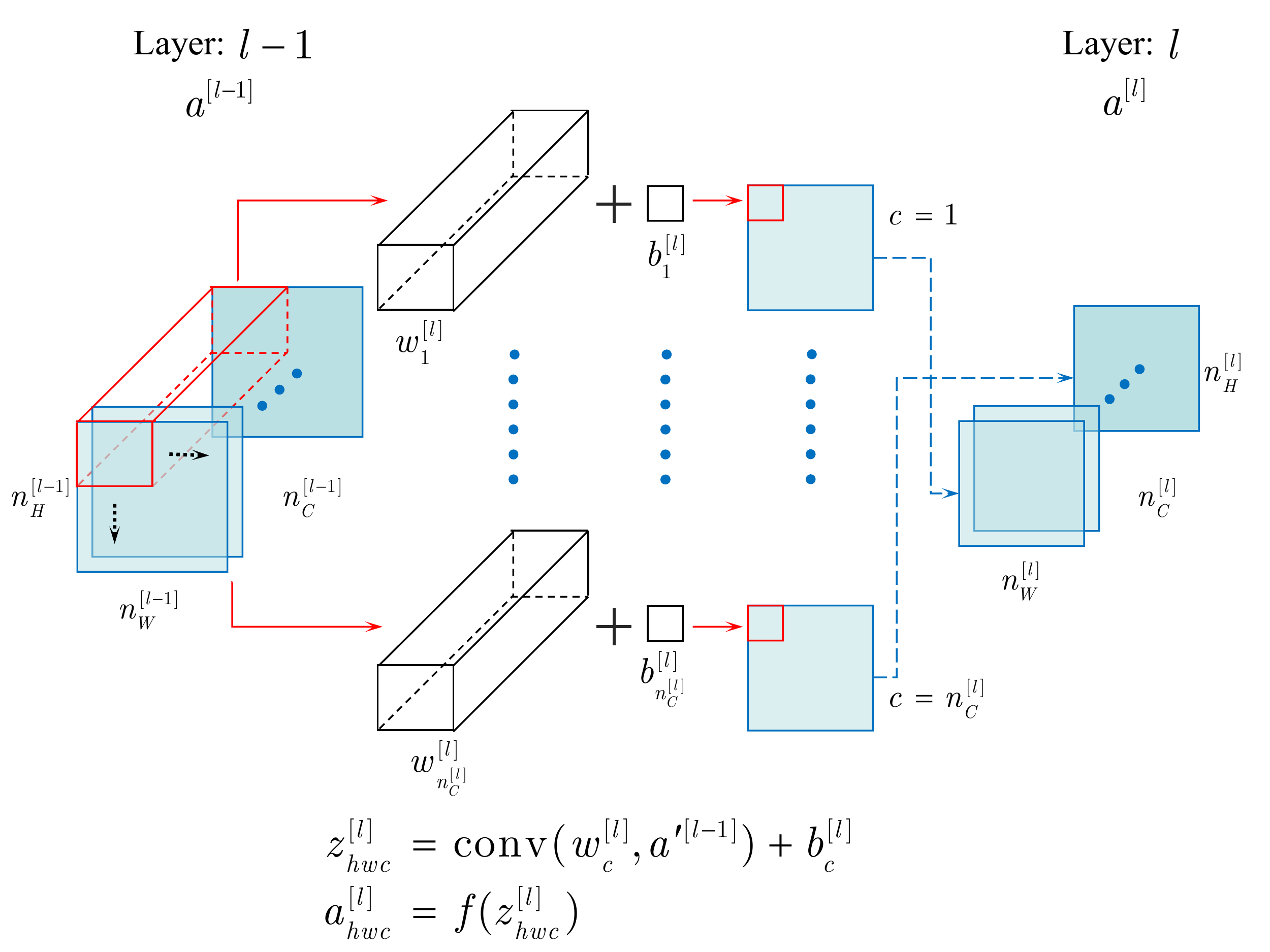}
  \caption{Forward propagation of convolution layers.}\label{Fig4}
\end{figure}

Similarly, the gradients of backward propagation are divided into two types. The transferring gradient of convolution layers is also kept the same as integer order gradient. Considering that the input is a batch of samples, the updating gradient is
\begin{eqnarray}\label{Eq3.9}
\left\{ \begin{array}{rl}
\frac{{{\partial ^\alpha }L}}{{\partial w_{ijkc}^{[l]\alpha }}} = & \hspace{-8pt} \sum\limits_{s = 1}^m {\sum\limits_{h = 1}^{n_H^{[l]}} {\sum\limits_{w = 1}^{n_W^{[l]}} {\frac{{\partial L}}{{\partial A_{shwc}^{[l]}}}} } } \frac{{\partial A_{shwc}^{[l]}}}{{\partial Z_{shwc}^{[l]}}}\frac{{{\partial ^\alpha }Z_{shwc}^{[l]}}}{{\partial w_{ijkc}^{[l]\alpha }}}\\
= & \hspace{-8pt} \sum\limits_{s = 1}^m {\sum\limits_{h = 1}^{n_H^{[l]}} {\sum\limits_{w = 1}^{n_W^{[l]}} {\frac{{\partial L}}{{\partial A_{shwc}^{[l]}}}} } } {f^{(1)}}(Z_{shwc}^{[l]})\frac{{{\partial ^\alpha }Z_{shwc}^{[l]}}}{{\partial w_{ijkc}^{[l]\alpha }}},\\
\frac{{{\partial ^\alpha }L}}{{\partial b_c^{[l]\alpha }}} = & \hspace{-8pt} \sum\limits_{s = 1}^m {\sum\limits_{h = 1}^{n_H^{[l]}} {\sum\limits_{w = 1}^{n_W^{[l]}} {\frac{{\partial L}}{{\partial A_{shwc}^{[l]}}}} } } \frac{{\partial A_{shwc}^{[l]}}}{{\partial Z_{shwc}^{[l]}}}\frac{{{\partial ^\alpha }Z_{shwc}^{[l]}}}{{\partial b_c^{[l]\alpha }}}\\
= & \hspace{-8pt} \sum\limits_{s = 1}^m {\sum\limits_{h = 1}^{n_H^{[l]}} {\sum\limits_{w = 1}^{n_W^{[l]}} {\frac{{\partial L}}{{\partial A_{shwc}^{[l]}}}} } } {f^{(1)}}(Z_{shwc}^{[l]})\frac{{{\partial ^\alpha }Z_{shwc}^{[l]}}}{{\partial b_c^{[l]\alpha }}},
\end{array} \right.
\end{eqnarray}
where $W^{[l]} = [w^{[l]}_{ijkc}]\in \mathbb R^{F^{[l]} \times F^{[l]} \times n_C^{[l-1]} \times n_C^{[l]}}$ is the weight that contains all $w^{[l]}_c$, $A^{[l]} \in \mathbb R^{m \times n_H^{[l]} \times n_W^{[l]} \times n_C^{[l]}}$ is the $a^{[l]}$ over all  $m$ samples.

\noindent When the fractional order gradient method (\ref{Eq2.7}) is introduced, the updating gradient at the $K$-th iteration is changed to
\begin{eqnarray}\label{Eq3.10}
\left\{ \begin{array}{rl}
\frac{{{\partial ^\alpha }L}}{{\partial w_{ijkc}^{[l]\alpha }}} = & \hspace{-8pt} \sum\limits_{s = 1}^m \sum\limits_{h = 1}^{n_H^{[l]}} \sum\limits_{w = 1}^{n_W^{[l]}} [\frac{{\partial L}}{{\partial A_{shwc~(K-1)}^{~[l]}}} {f^{(1)}}(Z_{shwc~(K-1)}^{~[l]}) \\
& \hspace{-8pt} \frac{{{A'}_{ijk~(K - 1)}^{~[l - 1]}}}{\Gamma(2-\alpha)} {|w_{ijkc~(K)}^{[l]} - w_{ijkc~(K - 1)}^{[l]}{|^{1 - \alpha }}}],\\
\frac{{{\partial ^\alpha }L}}{{\partial b_c^{[l]\alpha }}} = & \hspace{-8pt} \sum\limits_{s = 1}^m \sum\limits_{h = 1}^{n_H^{[l]}} \sum\limits_{w = 1}^{n_W^{[l]}} [\frac{{\partial L}}{{\partial A_{shwc~(K-1)}^{~[l]}}} {f^{(1)}}(Z_{shwc~(K-1)}^{~[l]}) \\
& \hspace{-8pt} {\frac{1}{{\Gamma (2 - \alpha )}}|b_{c~(K)}^{[l]} - b_{c~(K - 1)}^{[l]}{|^{1 - \alpha }}}],
\end{array} \right.
\end{eqnarray}
where ${A'}^{[l - 1]}$ is ${a'}^{[l - 1]}$ of the $s$-th sample. It could be simply regarded as ${A'}^{[l - 1]} = A^{[l - 1]}{[s, V_{start} : V_{end}, H_{start} : H_{end}]} \in \mathbb R^{F^{[l]} \times F^{[l]} \times n_C^{[l-1]}}$ with
\[V_{start}=(h-1) \times stride + 1, ~V_{end}=V_{start} + F^{[l]}, \]
\[H_{start}=(w-1) \times stride + 1, ~H_{end}=H_{start} + F^{[l]},\]
and $stride$ is moving length of convolution kernel each time. After vectorization, (\ref{Eq3.10}) is further simplified as
{\color{b}
\begin{eqnarray}\label{Eq3.11}
\hspace{-8pt} \left\{ \begin{array}{rl}
\frac{{{\partial ^\alpha }L}}{{\partial w_{c}^{[l]\alpha }}} = & \hspace{-8pt} \sum\limits_{s = 1}^m \sum\limits_{h = 1}^{n_H^{[l]}} \sum\limits_{w = 1}^{n_W^{[l]}} [\frac{1}{\Gamma(2-\alpha)} \frac{{\partial L}}{{\partial A_{shwc~(K-1)}^{~[l]}}} {f^{(1)}}(Z_{shwc~(K-1)}^{~[l]}) \\
& \hspace{-8pt} {A'}_{~(K - 1)}^{~[l - 1]} \circ |w_{c~(K)}^{[l]} - w_{c~(K - 1)}^{[l]}|^{1 - \alpha }],\\
\frac{{{\partial ^\alpha }L}}{{\partial b_c^{[l]\alpha }}} = & \hspace{-8pt} \sum\limits_{s = 1}^m \sum\limits_{h = 1}^{n_H^{[l]}} \sum\limits_{w = 1}^{n_W^{[l]}} [\frac{1}{{\Gamma (2 - \alpha )}} \frac{{\partial L}}{{\partial A_{shwc~(K-1)}^{~[l]}}} {f^{(1)}}(Z_{shwc~(K-1)}^{~[l]}) \\
& \hspace{-8pt} |b_{c~(K)}^{[l]} - b_{c~(K - 1)}^{[l]}|^{1 - \alpha}],
\end{array} \right.
\end{eqnarray}
}

\noindent In order to show the algorithm clearly, the calculation of (\ref{Eq3.11}) is transformed to following process.

\begin{algorithm}[htb]
\caption{Backward propagation of convolution layers by fractional order gradient method.}\label{Algorithm1}
\begin{algorithmic}[1]
\STATE {\color{b}$\frac{\partial L}{\partial Z^{[l]}_{~(K-1)}} = \frac{\partial L}{\partial A^{[l]}_{~(K-1)}} \circ f^{(1)}(Z^{[l]}_{~(K-1)})$}
\FOR{$s=1, 2, \cdots, m$}
\FOR{$h=1, 2, \cdots, n^{[l]}_H$}
\FOR{$w=1, 2, \cdots, n^{[l]}_W$}
\STATE {\color{b} $V_{start}=(h-1) \times stride + 1, ~V_{end}=V_{start} + F^{[l]}$}
\STATE {\color{b} $H_{start}=(w-1) \times stride + 1, ~H_{end}=H_{start} + F^{[l]}$}
\FOR{$c=1, 2, \cdots, n^{[l]}_C$}
\STATE ${A'}_{~(K - 1)}^{~[l - 1]} = {A}_{~(K - 1)}^{~[l - 1]}[s, V_{start} : V_{end}, H_{start} : H_{end}]$
\STATE {\color{b}$\frac{\partial ^\alpha L}{\partial w^{[l]\alpha}_{c}} += \frac{\partial L}{\partial Z^{[l]}_{shwc~(K-1)}} \frac{{A'}_{~(K - 1)}^{~[l - 1]}}{\Gamma (2-\alpha)} \circ |w^{[l]}_{c~(K)} - w^{[l]}_{c~(K-1)}|^{1-\alpha}$}
\STATE {\color{b}$\frac{\partial ^\alpha L}{\partial b^{[l]\alpha}_{c}} += \frac{1}{\Gamma (2-\alpha)} \frac{\partial L}{\partial Z^{[l]}_{shwc~(K-1)}} |b^{[l]}_{c~(K)} - b^{[l]}_{c~(K-1)}|^{1-\alpha}$}
\STATE $\frac{\partial L}{\partial {A'}^{[l-1]}_{~(K)}} = \frac{\partial L}{\partial Z^{[l]}_{shwc~(K)}} w^{[l]}_{c~(K)}$
\STATE $\frac{\partial L}{\partial A^{[l-1]}_{~(K)}[s, V_{start} : V_{end}, H_{start} : H_{end}]} += \frac{\partial L}{\partial {A'}^{[l-1]}_{~(K)}}$
\ENDFOR
\ENDFOR
\ENDFOR
\ENDFOR
\STATE \textbf{return} $\frac{\partial ^\alpha L}{\partial W^{[l]\alpha}}$, $\frac{\partial ^\alpha L}{\partial b^{[l]\alpha}}$, $\frac{\partial L}{\partial A^{[l-1]}_{~(K)}}$
\end{algorithmic}
\end{algorithm}

\noindent Then the updating parameters of convolution layers is as follows
\begin{eqnarray}\label{Eq3.12}
\left\{ \begin{array}{rl}
W^{[l]}_{~(K+1)} = & \hspace{-8pt} W^{[l]}_{~(K)} - \mu ~\frac{\partial {}^\alpha L}{\partial W^{[l]} {}^\alpha},\\
b^{[l]}_{~(K+1)} = & \hspace{-8pt} b^{[l]}_{~(K)} - \mu  ~\frac{\partial {}^\alpha L}{\partial b^{[l]} {}^\alpha}.
\end{array} \right.
\end{eqnarray}

\begin{theorem}\label{Theorem3}
The convolution layers updated by fractional order gradient method (\ref{Eq3.11}, \ref{Eq3.12}) are convergent to real extreme point.
\end{theorem}

{\color{b}
The proof resembles Theorem \ref{Theorem2}. By introducing integer order gradients, the gradients in (\ref{Eq3.11}) are changed into the form like (\ref{Eq3.82}). Then it is a proof by contradiction that could be done for each element of $W^{[l]}$ and $b^{[l]}$ in (\ref{Eq3.12}).
}

\begin{remark}\label{Remark3}
Based on backward propagation of convolution layers, when padding is introduced into the $l$-th layers, the transferring gradient will be influenced. The gradient $\frac{\partial L}{\partial A^{[l-1]}}$ calculated by Algorithm \ref{Algorithm1} is the gradient of padded output. The padded part of $\frac{\partial L}{\partial A^{[l-1]}}$ needs deleting. However, there is no change happened for the updating gradient of fractional order.
\end{remark}

\begin{remark}\label{Remark4}
During the training procedure, a tiny value could be added to (\ref{Eq3.7}, \ref{Eq3.11}) so that the singularity caused by $W^{[l]}_{~(K)} = W^{[l]}_{~(K-1)}$ or $b^{[l]}_{~(K)} = b^{[l]}_{~(K-1)}$ is avoided easily. Hence the gradients modified by (\ref{Eq2.11}) are listed below
{\color{b}
\begin{eqnarray}\label{Eq3.13}
\left\{ \begin{array}{rl}
\frac{\partial {}^\alpha L}{\partial W^{[l]} {}^\alpha} = & \hspace{-8pt} \frac{1}{\Gamma(2-\alpha)} \left[\frac{\partial L}{\partial A^{[l]}_{~(K-1)}} \circ f^{(1)}(Z^{[l]}_{~(K-1)})\right] A^{[l-1] ~\rm T}_{~(K-1)} \\
& \hspace{-8pt} \circ |W_{~(K)}^{[l]} - W_{~{(K-1)}}^{[l]} + \delta|^{1 - \alpha}, \\
\frac{\partial {}^\alpha L}{\partial b^{[l]} {}^\alpha} = & \hspace{-8pt} \frac{1}{\Gamma(2-\alpha)} {\rm sum}\left(\frac{\partial L}{\partial A^{[l]}_{~(K-1)}} \circ f^{(1)}(Z^{[l]}_{~(K-1)})\right) \\
& \hspace{-8pt} \circ |b_{~(K)}^{[l]} - b_{~{(K-1)}}^{[l]} + \delta|^{1 - \alpha},\\
\end{array} \right.
\end{eqnarray}
\begin{eqnarray}\label{Eq3.14}
\hspace{-8pt} \left\{ \begin{array}{rl}
\frac{{{\partial ^\alpha }L}}{{\partial w_{c}^{[l]\alpha }}} = & \hspace{-8pt} \sum\limits_{s = 1}^m \sum\limits_{h = 1}^{n_H^{[l]}} \sum\limits_{w = 1}^{n_W^{[l]}} [\frac{1}{\Gamma(2-\alpha)} \frac{{\partial L}}{{\partial A_{shwc~(K-1)}^{~[l]}}} {f^{(1)}}(Z_{shwc~(K-1)}^{~[l]}) \\
& \hspace{-8pt} {A'}_{~(K - 1)}^{~[l - 1]} \circ |w_{c~(K)}^{[l]} - w_{c~(K - 1)}^{[l]} + \delta|^{1 - \alpha }],\\
\frac{{{\partial ^\alpha }L}}{{\partial b_c^{[l]\alpha }}} = & \hspace{-8pt} \sum\limits_{s = 1}^m \sum\limits_{h = 1}^{n_H^{[l]}} \sum\limits_{w = 1}^{n_W^{[l]}} [\frac{1}{{\Gamma (2 - \alpha )}} \frac{{\partial L}}{{\partial A_{shwc~(K-1)}^{~[l]}}} {f^{(1)}}(Z_{shwc~(K-1)}^{~[l]}) \\
& \hspace{-8pt} |b_{c~(K)}^{[l]} - b_{c~(K - 1)}^{[l]} + \delta|^{1 - \alpha}].
\end{array} \right.
\end{eqnarray}
}
\end{remark}

When convolution layers, pooling layers and fully connected layers are connected end-to-end, it will fulfill some tasks like classification. Many types of functions, such as quadratic function $L=(y-\hat y)^2$ and cross-entropy function $L=-y{\rm log}(\hat y)-(1-y){\rm log}(1-\hat y)$, are available to loss function of proposed method. And the gradient on output of the last layer $\frac{\partial L}{\partial \hat y}$ is still integer order. The backward propagation between layers of different types is simple. Reshaping is needed only when the gradients correspond to different shapes. Finally, a type of complete convolutional neural networks will combine with fractional order gradient to fulfill a task and demonstrate the effectiveness of proposed method. For example, the LeNet \cite{LeCun:1998IEEE} is a simple CNN which will be adopted for subsequent experiment.

\section{Experiments}\label{Section 4}
\begin{figure*}[htb]
  \centering
  \includegraphics[width=0.8\hsize]{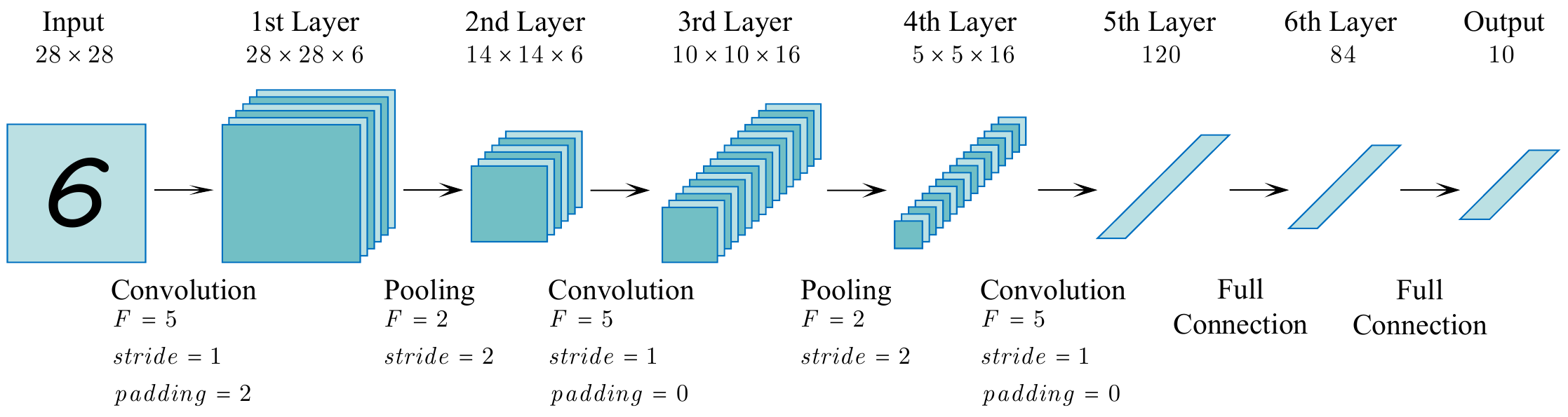}
  \caption{Architecture of LeNet-5 with input of MNIST dataset.}\label{Fig5}
\end{figure*}
The task of experiments is to identify handwriting number by fractional order convolutional neural networks. The experiments are carried out by the MNIST dataset which consists of 60,000 handwritten digit images for the training and another 10,000 samples for testing. Consequently, the whole structure of LeNet is presented in Fig. \ref{Fig5}.

\begin{figure}[htb]
  \centering
  \includegraphics[width=0.8\hsize]{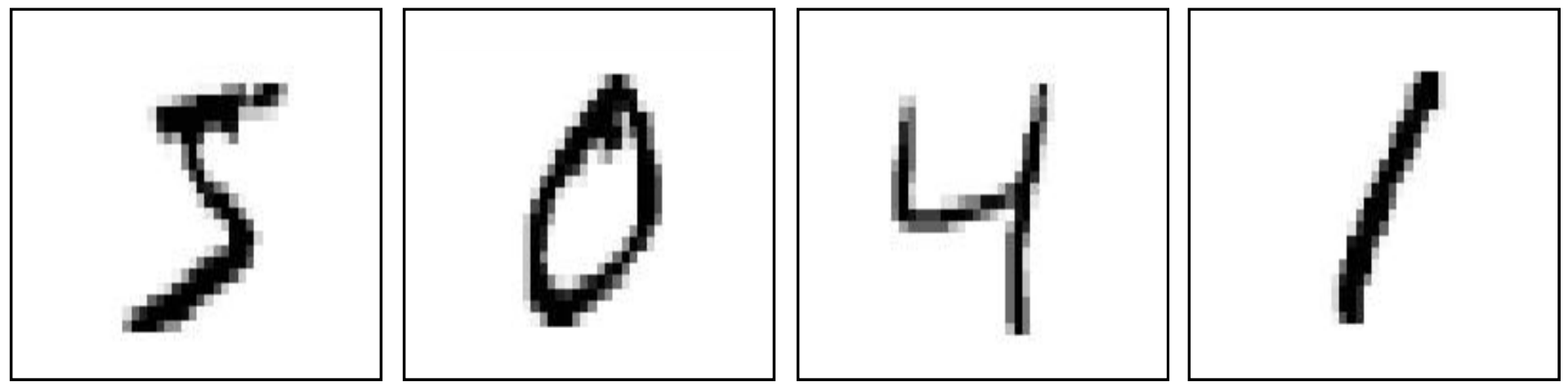}
  \caption{Some samples of MNIST dataset.}\label{Fig6}
\end{figure}

\noindent The corresponding parameters are listed below. ${\hat y}_s \in {\mathbb R}^{10}$ is the output of networks for the $s$-th smaple and $y_s \in {\mathbb R}^{10}$ is label with one-hot form.
\begin{table}[htb]
\renewcommand\arraystretch{1.3}
\centering
\label{Tab1}
\begin{tabular}{ll}
\hline
Loss function: & $L=-\frac{1}{m}\sum\limits_{s = 1}^m y_s^{\rm T}{\rm log}({\hat y}_s)$ \\
Learning rate: & $\mu = 0.1$ \\
Batch size: & $m = 10$ \\
Initial weight: & $w \in [-0.1, 0.1]$ \\
Initial bias: & $b \in [-0.1, 0.1]$ \\
Number of iteration: & $Iteration = 6000$ \\
Number of epoch: & $Epoch = 1$ \\
\hline
\end{tabular}
\end{table}

\noindent The experiments are carried out by 10 times. All parameters, such as weights, bias and the inputting order of samples, are randomly initialized each time. Consequently, the training accuracy and testing accuracy with different fractional order are shown in Table \ref{Tab1}.
\begin{table}[hbt]
\renewcommand\arraystretch{1.3}
\centering
\caption{{\color{b}The average training and testing accuracy.}}\label{Tab1}
\begin{tabular}{ccc}
\hline
\hline
 order($\alpha$) & training accuracy & testing accuracy \\
\hline
 1.9 & 0.0980 & 0.1006 \\
 1.8 & 0.0978 & 0.1009 \\
 1.7 & 0.0978 & 0.1009 \\
 1.6 & 0.0978 & 0.1009 \\
 1.5 & 0.2556 & 0.2581 \\
 1.4 & 0.5924 & 0.5938 \\
 1.3 & 0.8739 & 0.8741 \\
 1.2 & 0.9734 & 0.9728 \\
 1.1 & 0.9813 & 0.9803 \\
 1.0 & 0.9783 & 0.9781 \\
 0.9 & 0.9805 & 0.9799 \\
 0.8 & 0.9780 & 0.9767 \\
 0.7 & 0.9724 & 0.9711 \\
 0.6 & 0.9646 & 0.9637 \\
 0.5 & 0.9498 & 0.9516 \\
 0.4 & 0.9267 & 0.9322 \\
 0.3 & 0.8913 & 0.9004 \\
 0.2 & 0.6671 & 0.6759 \\
 0.1 & 0.3052 & 0.3050 \\
\hline
\hline
\end{tabular}
\end{table}

It could be observed that the accuracy of fractional order gradient methods with $\alpha=0.9$ and $\alpha=1.1$ is higher than the integer order gradient method in most cases. What's more, when average accuracy of 10 experiments is taken into consideration, the integer order one shows a little less accuracy. In Fig. \ref{Fig7}, the average accuracy of training and testing results is drawn over $\alpha=0.1, \cdots, 1.9$.

\begin{figure}[htb]
  \centering
  \includegraphics[width=0.9\hsize]{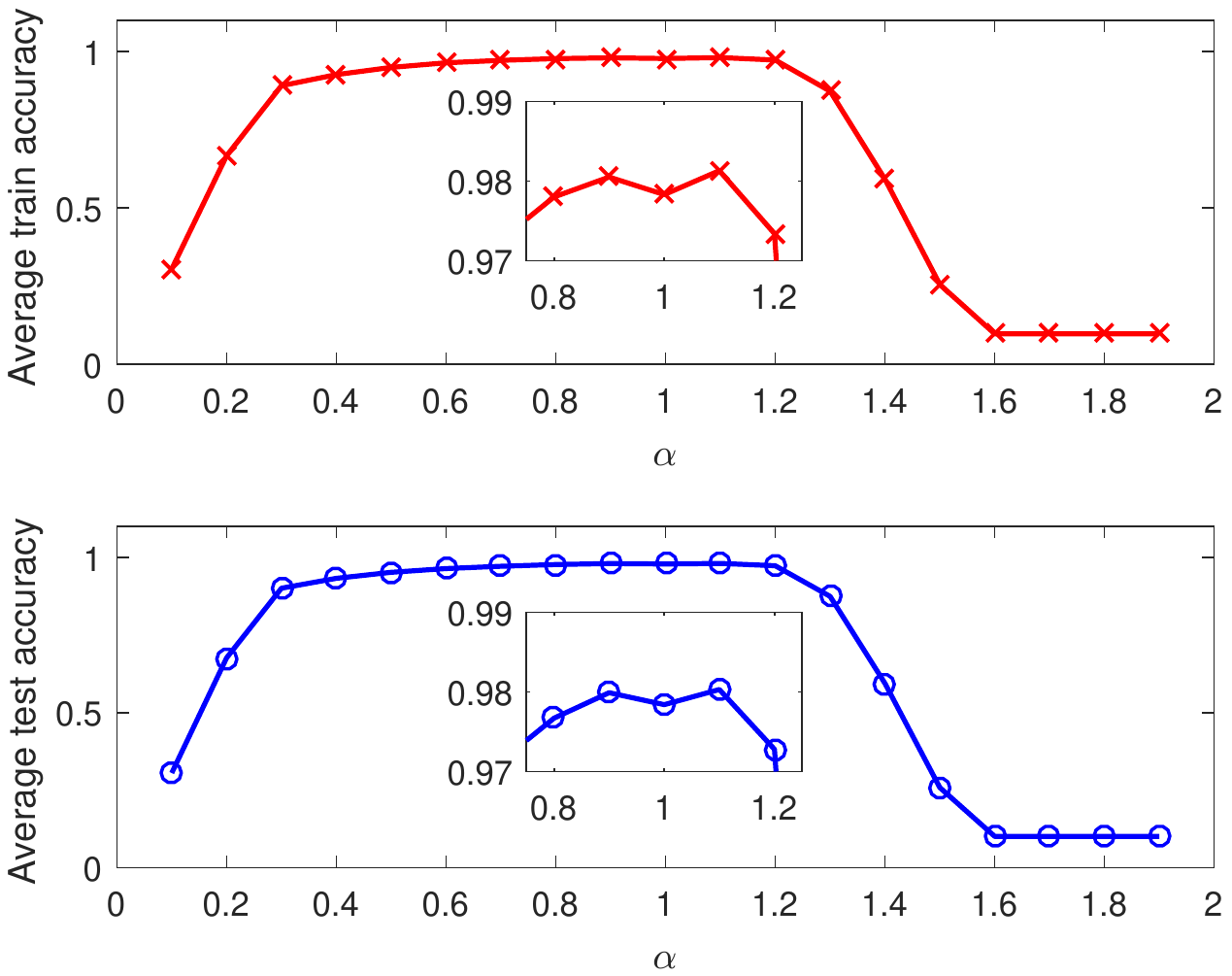}
  \caption{Average accuracy of training and testing results.}\label{Fig7}
\end{figure}

\begin{figure}[htb]
  \centering
  \includegraphics[width=0.9\hsize]{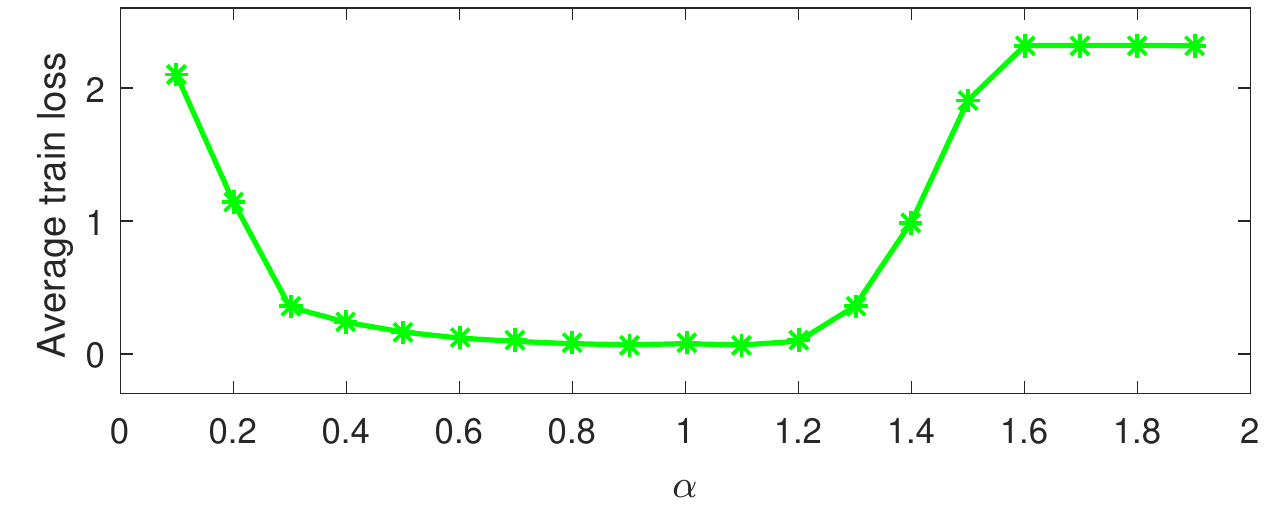}
  \caption{Average loss of training results.}\label{Fig8}
\end{figure}

Although fractional order gradient method works well in CNN, it is not effective enough all the time when $\alpha<0.3$ or $\alpha>1.2$. The reason of such low accuracy is caused by the the Gamma function in fractional order calculus (\ref{Eq2.1}). The Gamma function $\Gamma(2-\alpha)$ in fractional order gradient method  (\ref{Eq3.7}, \ref{Eq3.11}) is a very large number for $\alpha<0.3$ or $\alpha>1.2$. As a result, the gradients are too small to reduce the loss function and sink into a local extreme point quickly. Since it is often a point close to the initial point, the loss does not decrease or only decreases a little bit. This phenomenon is also demonstrated by Fig. \ref{Fig8}.

\begin{figure}
  \centering
  \subfigure[$\alpha = 0.9$]{
  \label{Fig9a}
  \includegraphics[width=1.0\hsize]{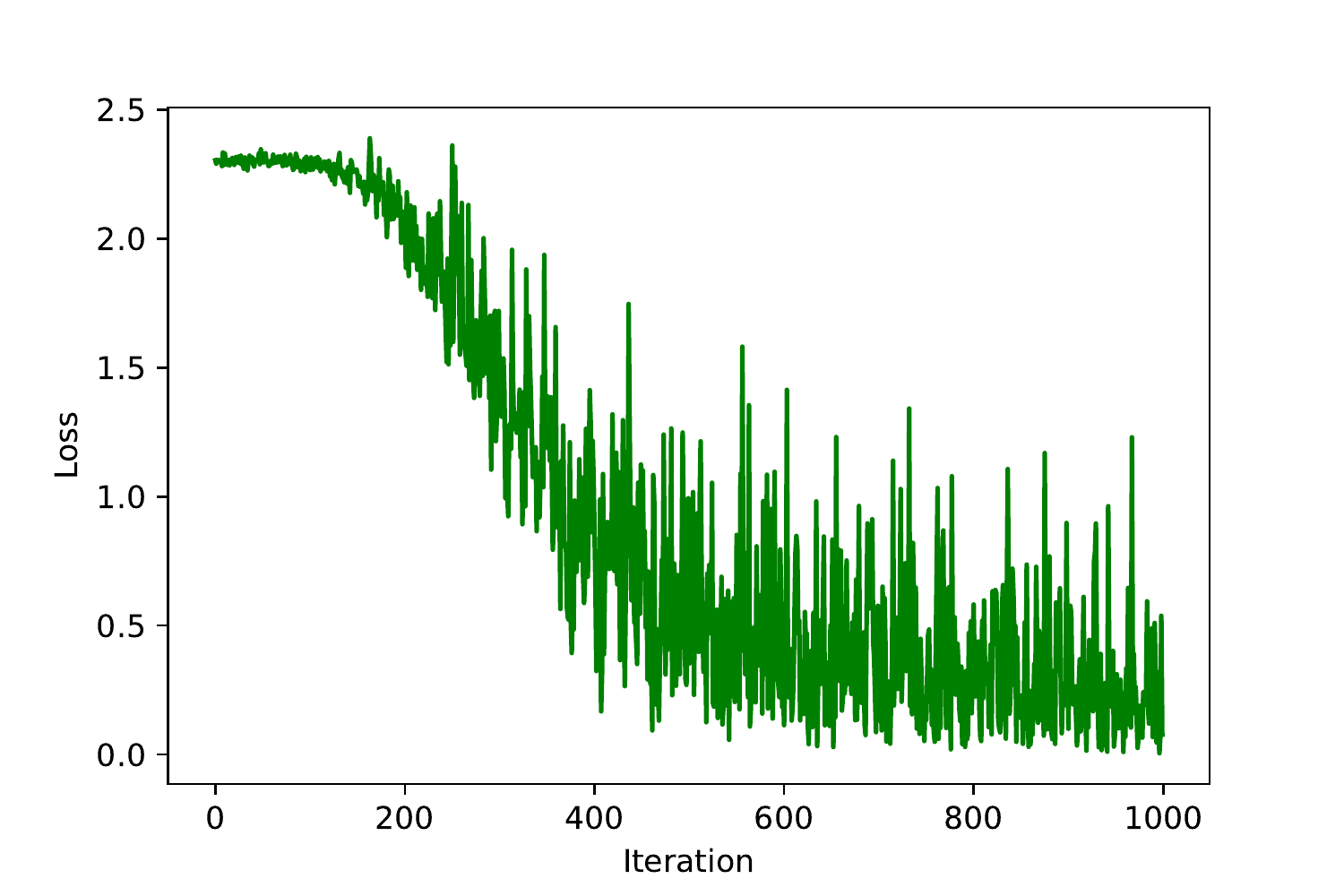}}

  \subfigure[$\alpha = 1.0$]{
  \label{Fig9b}
  \includegraphics[width=1.0\hsize]{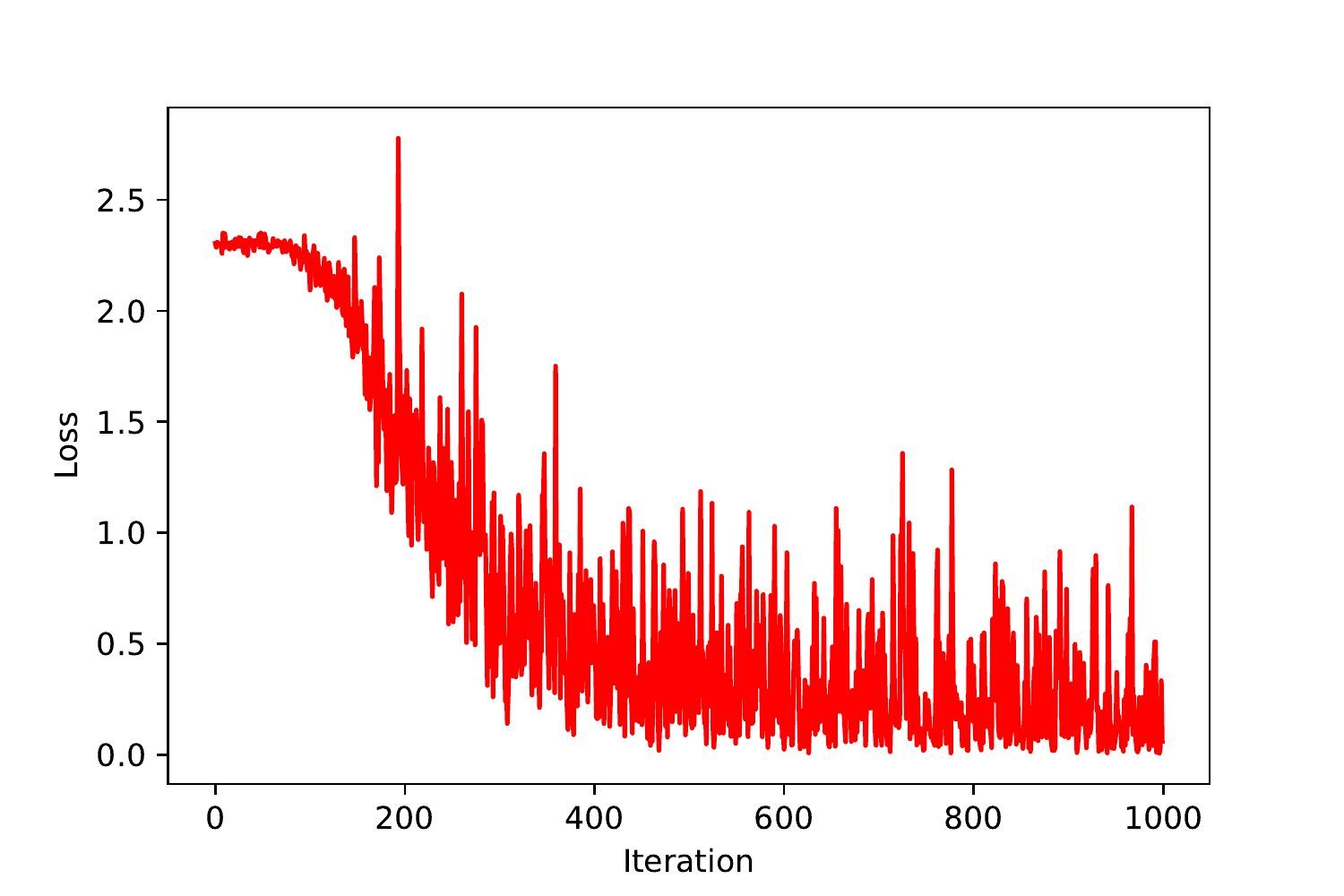}}

  \subfigure[$\alpha = 1.1$]{
  \label{Fig9c}
  \includegraphics[width=1.0\hsize]{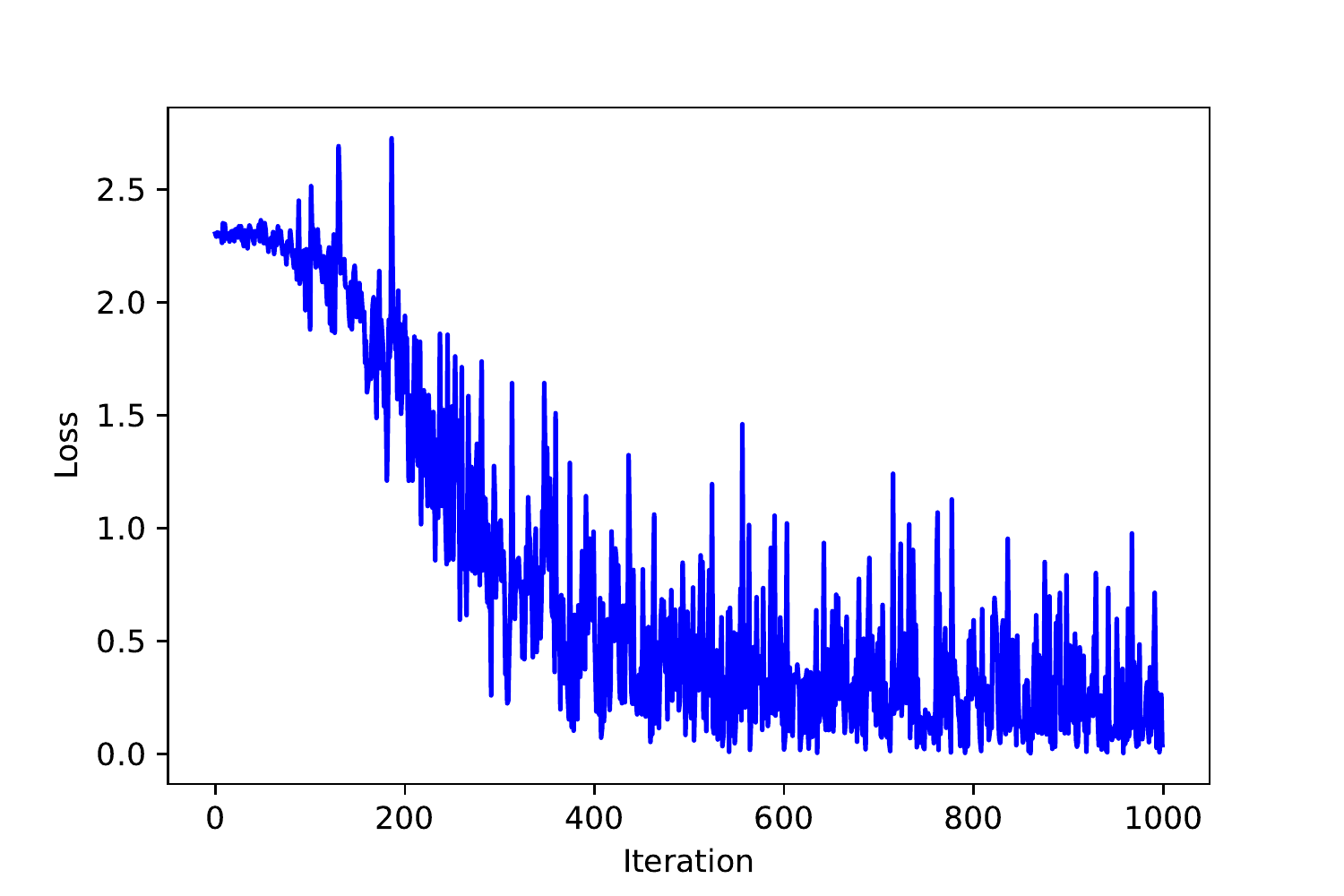}}
  \caption{Loss functions for different order.}
  \label{Fig9} 
\end{figure}

{\color{b} To analyze fractional order gradients further, the Fig. \ref{Fig9} is drawn here to show the average loss function of first 1000 iterations. Even if all loss functions are decreasing during training, the loss iterated by fractional order gradients seems to prefer jumping farther. Compared with integer order gradient method in Fig. \ref{Fig9b}, the points are more dispersed for fractional order cases, especially for $\alpha=0.9$ in Fig. \ref{Fig9a}. In addition, each variance of different loss function is listed below.
\begin{table}[hbt]
\renewcommand\arraystretch{1.3}
\centering
\caption{{\color{b}The variance of loss function.}}\label{Tab2}
\begin{tabular}{cccc}
\hline
\hline
 order($\alpha$) & \hspace{5pt} 0.9 & \hspace{10pt} 1.0 & \hspace{10pt} 1.1 \\
\hline
 variance & \hspace{5pt} 0.65370 & \hspace{10pt} 0.56961 & \hspace{10pt} 0.57697 \\
\hline
\hline
\end{tabular}
\end{table}

\noindent The larger variance often indicates looser distribution, which implies that fractional order gradient method helps optimizing process jump frequently and far. Therefore, it is provided with more possibility to escape the local optimal point.
}

{\color{b}In view of seemingly complicated calculation, it seems that fractional order gradient method needs more time than integer order one.} However, the training speed of fractional order CNN is almost as fast as integer order CNN. Taking all experiments into consideration, the average training time spent by integer order CNN is only 0.53\% less than fractional order CNN with $\alpha=1.1$. Similar speed also exists in other cases for $\alpha=0.1, \cdots, 1.9$. There are two reasons that result in such fast speed of fractional order gradient method for CNN. One reason is that only updating gradients are replaced by fractional order. The other reason is that the fractional order updating gradients are obtained according to integer order gradients and the additional calculation in fractional order updating gradients are quite simple.

\section{Conclusions}\label{Section 5}
The backward propagation of neural networks is investigated by fractional order gradient method in this paper. After modification of fractional order gradient, the proposed gradient method can ensure the convergence to real extreme point, and has been successfully applied in CNN for updating parameters. It is the first time for CNN to cooperate with fractional order calculus. The chain rule {\color{b}in backward propagation }is completely preserved since integer order gradients are still used for transferring between layers. Both the range of fractional order and the type of loss function are enlarged. {\color{b}Moreover, the proposed fractional order gradient method verifies its fast convergence, high accuracy and ability to escape local optimal point in neural networks when compared with integer order case.} It is believed that this paper provides a new way to study gradient method and its application.

\section*{Acknowledgements}
\phantomsection
\addcontentsline{toc}{section}{Acknowledgements}
The work described in this paper was fully supported by the National Natural Science Foundation of China (No. 61573332, No. 61601431), the Fundamental Research Funds for the Central Universities (No. WK2100100028), the Anhui Provincial Natural Science Foundation (No. 1708085QF141) and the General Financial Grant from the China Postdoctoral Science Foundation (No. 2016M602032).

\section*{References}
\phantomsection
\addcontentsline{toc}{section}{References}
\bibliographystyle{model3-num-names}
\bibliography{database}

\begin{thebibliography}{26}
\providecommand{\natexlab}[1]{#1}
\providecommand{\url}[1]{\texttt{#1}}
\providecommand{\href}[2]{#2}
\providecommand{\path}[1]{#1}
\providecommand{\eprint}[1]{\href{http://arxiv.org/abs/#1}{\path{#1}}}
\providecommand{\DOIprefix}{doi:}
\providecommand{\ArXivprefix}{arXiv:}
\providecommand{\URLprefix}{URL: }
\providecommand{\Pubmedprefix}{pmid:}
\providecommand{\doi}[1]{\href{http://dx.doi.org/#1}{\path{#1}}}
\providecommand{\Pubmed}[1]{\href{pmid:#1}{\path{#1}}}
\providecommand{\BIBand}{and}
\providecommand{\bibinfo}[2]{#2}
\ifx\xfnm\undefined \def\xfnm[#1]{\unskip,\space#1}\fi
\bibitem[{LeCun et~al.(2015)LeCun, Bengio and Hinton}]{LeCun:2015Nature}
\bibinfo{author}{LeCun\xfnm[ Y.]}, \bibinfo{author}{Bengio\xfnm[ Y.]},
  \bibinfo{author}{Hinton\xfnm[ G.]}.
\newblock \bibinfo{title}{Deep learning}.
\newblock \bibinfo{journal}{Nature}
  \bibinfo{year}{2015};\bibinfo{volume}{521}(\bibinfo{number}{7553}):\bibinfo{pages}{436}.
\bibitem[{Fukushima(1980)}]{Fukushima:1980BC}
\bibinfo{author}{Fukushima\xfnm[ K.]}.
\newblock \bibinfo{title}{Neocognitron: A self-organizing neural network model
  for a mechanism of pattern recognition unaffected by shift in position}.
\newblock \bibinfo{journal}{Biological Cybernetics}
  \bibinfo{year}{1980};\bibinfo{volume}{36}(\bibinfo{number}{4}):\bibinfo{pages}{193--202}.
\bibitem[{Rumelhart et~al.(1988)Rumelhart, McClelland, Group
  et~al.}]{Rumelhart:1988Book}
\bibinfo{author}{Rumelhart\xfnm[ D.E.]}, \bibinfo{author}{McClelland\xfnm[
  J.L.]}, \bibinfo{author}{Group\xfnm[ P.R.]}, et~al.
\newblock \bibinfo{title}{Parallel distributed processing}.
\newblock \bibinfo{address}{Boston}: \bibinfo{publisher}{MIT press Cambridge};
  \bibinfo{year}{1988}.
\bibitem[{Alexander et~al.(1990)Alexander, Hanazawa, Hinton, Kiyohiro and
  Lang}]{Alexander1990RSR}
\bibinfo{author}{Alexander\xfnm[ W.]}, \bibinfo{author}{Hanazawa\xfnm[ T.]},
  \bibinfo{author}{Hinton\xfnm[ G.]}, \bibinfo{author}{Kiyohiro\xfnm[ S.]},
  \bibinfo{author}{Lang\xfnm[ K.]}.
\newblock \bibinfo{title}{Phoneme recognition using time-delay neural
  networks}.
\newblock \bibinfo{journal}{Readings in Speech Recognition}
  \bibinfo{year}{1990};\bibinfo{volume}{1}(\bibinfo{number}{2}):\bibinfo{pages}{393--404}.
\bibitem[{LeCun et~al.(1989)LeCun, Boser, Denker, Henderson, Howard, Hubbard
  et~al.}]{LeCun:1989NC}
\bibinfo{author}{LeCun\xfnm[ Y.]}, \bibinfo{author}{Boser\xfnm[ B.]},
  \bibinfo{author}{Denker\xfnm[ J.S.]}, \bibinfo{author}{Henderson\xfnm[ D.]},
  \bibinfo{author}{Howard\xfnm[ R.E.]}, \bibinfo{author}{Hubbard\xfnm[ W.]},
  et~al.
\newblock \bibinfo{title}{Backpropagation applied to handwritten zip code
  recognition}.
\newblock \bibinfo{journal}{Neural Computation}
  \bibinfo{year}{1989};\bibinfo{volume}{1}(\bibinfo{number}{4}):\bibinfo{pages}{541--551}.
\bibitem[{LeCun et~al.(1998)LeCun, Bottou, Bengio, Haffner
  et~al.}]{LeCun:1998IEEE}
\bibinfo{author}{LeCun\xfnm[ Y.]}, \bibinfo{author}{Bottou\xfnm[ L.]},
  \bibinfo{author}{Bengio\xfnm[ Y.]}, \bibinfo{author}{Haffner\xfnm[ P.]},
  et~al.
\newblock \bibinfo{title}{Gradient-based learning applied to document
  recognition}.
\newblock \bibinfo{journal}{Proceedings of the IEEE}
  \bibinfo{year}{1998};\bibinfo{volume}{86}(\bibinfo{number}{11}):\bibinfo{pages}{2278--2324}.
\bibitem[{Krizhevsky et~al.(2012)Krizhevsky, Sutskever and
  Hinton}]{Krizhevsky:2012imagenet}
\bibinfo{author}{Krizhevsky\xfnm[ A.]}, \bibinfo{author}{Sutskever\xfnm[ I.]},
  \bibinfo{author}{Hinton\xfnm[ G.E.]}.
\newblock \bibinfo{title}{Imagenet classification with deep convolutional
  neural networks}.
\newblock In: \bibinfo{booktitle}{2012 Neural Information Processing Systems
  (NIPS)}. \bibinfo{address}{Lake Tahoe,~USA}; \bibinfo{year}{2012}, p.
  \bibinfo{pages}{1097--1105}.
\bibitem[{Simonyan and Zisserman(2015)}]{Simonyan:2015ICLP}
\bibinfo{author}{Simonyan\xfnm[ K.]}, \bibinfo{author}{Zisserman\xfnm[ A.]}.
\newblock \bibinfo{title}{Very deep convolutional networks for large-scale
  image recognition}.
\newblock In: \bibinfo{booktitle}{2015 International Conference on Learning
  Representations (ICLR)}. \bibinfo{address}{San Diego, USA};
  \bibinfo{year}{2015}, p. \bibinfo{pages}{1--14}.
\bibitem[{Szegedy et~al.(2015)Szegedy, Liu, Jia, Sermanet, Reed, Anguelov
  et~al.}]{Szegedy:2015CVPR}
\bibinfo{author}{Szegedy\xfnm[ C.]}, \bibinfo{author}{Liu\xfnm[ W.]},
  \bibinfo{author}{Jia\xfnm[ Y.]}, \bibinfo{author}{Sermanet\xfnm[ P.]},
  \bibinfo{author}{Reed\xfnm[ S.]}, \bibinfo{author}{Anguelov\xfnm[ D.]},
  et~al.
\newblock \bibinfo{title}{Going deeper with convolutions}.
\newblock In: \bibinfo{booktitle}{2015 IEEE Conference on Computer Vision and
  Pattern Recognition (CVPR)}. \bibinfo{address}{Boston, USA};
  \bibinfo{year}{2015}, p. \bibinfo{pages}{1--9}.
\bibitem[{He et~al.(2016)He, Zhang, Ren and Sun}]{He:2016CVPR}
\bibinfo{author}{He\xfnm[ K.M.]}, \bibinfo{author}{Zhang\xfnm[ X.Y.]},
  \bibinfo{author}{Ren\xfnm[ S.Q.]}, \bibinfo{author}{Sun\xfnm[ J.]}.
\newblock \bibinfo{title}{Deep residual learning for image recognition}.
\newblock In: \bibinfo{booktitle}{2016 IEEE Conference on Computer Vision and
  Pattern Recognition (CVPR)}. \bibinfo{address}{Las Vegas, USA};
  \bibinfo{year}{2016}, p. \bibinfo{pages}{770--778}.
\bibitem[{Raja and Chaudhary(2015)}]{Raja:2015SP}
\bibinfo{author}{Raja\xfnm[ M.A.Z.]}, \bibinfo{author}{Chaudhary\xfnm[ N.I.]}.
\newblock \bibinfo{title}{Two-stage fractional least mean square identification
  algorithm for parameter estimation of carma systems}.
\newblock \bibinfo{journal}{Signal Processing}
  \bibinfo{year}{2015};\bibinfo{volume}{107}:\bibinfo{pages}{327--339}.
\bibitem[{Cheng et~al.(2017)Cheng, Wei, Chen, Li and Wang}]{Cheng:2017SP}
\bibinfo{author}{Cheng\xfnm[ S.S.]}, \bibinfo{author}{Wei\xfnm[ Y.H.]},
  \bibinfo{author}{Chen\xfnm[ Y.Q.]}, \bibinfo{author}{Li\xfnm[ Y.]},
  \bibinfo{author}{Wang\xfnm[ Y.]}.
\newblock \bibinfo{title}{An innovative fractional order lms based on variable
  initial value and gradient order}.
\newblock \bibinfo{journal}{Signal Processing}
  \bibinfo{year}{2017};\bibinfo{volume}{133}:\bibinfo{pages}{260--269}.
\bibitem[{Yin et~al.(2019)Yin, Wei, Liu and Wang}]{Yin:2019MSSP}
\bibinfo{author}{Yin\xfnm[ W.D.]}, \bibinfo{author}{Wei\xfnm[ Y.H.]},
  \bibinfo{author}{Liu\xfnm[ T.Y.]}, \bibinfo{author}{Wang\xfnm[ Y.]}.
\newblock \bibinfo{title}{A novel orthogonalized fractional order filtered-x
  normalized least mean squares algorithm for feedforward vibration rejection}.
\newblock \bibinfo{journal}{Mechanical Systems and Signal Processing}
  \bibinfo{year}{2019};\bibinfo{volume}{119}:\bibinfo{pages}{138--154}.
\bibitem[{Cheng et~al.(2018)Cheng, Wei, Sheng, Chen and Wang}]{Cheng:2018SP}
\bibinfo{author}{Cheng\xfnm[ S.S.]}, \bibinfo{author}{Wei\xfnm[ Y.H.]},
  \bibinfo{author}{Sheng\xfnm[ D.]}, \bibinfo{author}{Chen\xfnm[ Y.]},
  \bibinfo{author}{Wang\xfnm[ Y.]}.
\newblock \bibinfo{title}{Identification for hammerstein nonlinear armax
  systems based on multi-innovation fractional order stochastic gradient}.
\newblock \bibinfo{journal}{Signal Processing}
  \bibinfo{year}{2018};\bibinfo{volume}{142}:\bibinfo{pages}{1--10}.
\bibitem[{Cui et~al.(2018)Cui, Wei, Cheng and Wang}]{Cui:2018ISAT}
\bibinfo{author}{Cui\xfnm[ R.Z.]}, \bibinfo{author}{Wei\xfnm[ Y.H.]},
  \bibinfo{author}{Cheng\xfnm[ S.S.]}, \bibinfo{author}{Wang\xfnm[ Y.]}.
\newblock \bibinfo{title}{An innovative parameter estimation for fractional
  order systems with impulse noise}.
\newblock \bibinfo{journal}{ISA Transactions}
  \bibinfo{year}{2018};\bibinfo{volume}{82}:\bibinfo{pages}{120--129}.
\bibitem[{Li et~al.(2009)Li, Chen and Podlubny}]{Li:2009Automatica}
\bibinfo{author}{Li\xfnm[ Y.]}, \bibinfo{author}{Chen\xfnm[ Y.Q.]},
  \bibinfo{author}{Podlubny\xfnm[ I.]}.
\newblock \bibinfo{title}{{Mittag-Leffler} stability of fractional order
  nonlinear dynamic systems}.
\newblock \bibinfo{journal}{Automatica}
  \bibinfo{year}{2009};\bibinfo{volume}{45}(\bibinfo{number}{8}):\bibinfo{pages}{1965--1969}.
\bibitem[{Lu and Chen(2010)}]{Lu:2010TAC}
\bibinfo{author}{Lu\xfnm[ J.G.]}, \bibinfo{author}{Chen\xfnm[ Y.Q.]}.
\newblock \bibinfo{title}{Robust stability and stabilization of
  fractional-order interval systems with the fractional order $\alpha$: the $0<
  \alpha <1$ case}.
\newblock \bibinfo{journal}{IEEE Transactions on Automatic Control}
  \bibinfo{year}{2010};\bibinfo{volume}{55}(\bibinfo{number}{1}):\bibinfo{pages}{152--158}.
\bibitem[{Yin et~al.(2014)Yin, Chen and Zhong}]{Yin:2014Automatica}
\bibinfo{author}{Yin\xfnm[ C.]}, \bibinfo{author}{Chen\xfnm[ Y.Q.]},
  \bibinfo{author}{Zhong\xfnm[ S.M.]}.
\newblock \bibinfo{title}{Fractional-order sliding mode based extremum seeking
  control of a class of nonlinear systems}.
\newblock \bibinfo{journal}{Automatica}
  \bibinfo{year}{2014};\bibinfo{volume}{50}(\bibinfo{number}{12}):\bibinfo{pages}{3173--3181}.
\bibitem[{Wei et~al.(2017)Wei, Du, Cheng and Wang}]{Wei:2017JOTA}
\bibinfo{author}{Wei\xfnm[ Y.H.]}, \bibinfo{author}{Du\xfnm[ B.]},
  \bibinfo{author}{Cheng\xfnm[ S.S.]}, \bibinfo{author}{Wang\xfnm[ Y.]}.
\newblock \bibinfo{title}{Fractional order systems time-optimal control and its
  application}.
\newblock \bibinfo{journal}{Journal of Optimization Theory and Applications}
  \bibinfo{year}{2017};\bibinfo{volume}{174}(\bibinfo{number}{1}):\bibinfo{pages}{122--138}.
\bibitem[{Pu et~al.(2015)Pu, Zhou, Zhang, Zhang, Huang and
  Siarry}]{Pu:2015TNNLS}
\bibinfo{author}{Pu\xfnm[ Y.F.]}, \bibinfo{author}{Zhou\xfnm[ J.L.]},
  \bibinfo{author}{Zhang\xfnm[ Y.]}, \bibinfo{author}{Zhang\xfnm[ N.]},
  \bibinfo{author}{Huang\xfnm[ G.]}, \bibinfo{author}{Siarry\xfnm[ P.]}.
\newblock \bibinfo{title}{Fractional extreme value adaptive training method:
  fractional steepest descent approach}.
\newblock \bibinfo{journal}{IEEE Transactions on Neural Networks and Learning
  Systems}
  \bibinfo{year}{2015};\bibinfo{volume}{26}(\bibinfo{number}{4}):\bibinfo{pages}{653--662}.
\bibitem[{Monje et~al.(2010)Monje, Chen, Vinagre, Xue and
  Feliu}]{Monje:2010Book}
\bibinfo{author}{Monje\xfnm[ C.A.]}, \bibinfo{author}{Chen\xfnm[ Y.Q.]},
  \bibinfo{author}{Vinagre\xfnm[ B.M.]}, \bibinfo{author}{Xue\xfnm[ D.Y.]},
  \bibinfo{author}{Feliu\xfnm[ V.]}.
\newblock \bibinfo{title}{Fractional-Order Systems and Controls: Fundamentals
  and Applications}.
\newblock \bibinfo{address}{London}: \bibinfo{publisher}{Springer};
  \bibinfo{year}{2010}.
\bibitem[{Chen et~al.(2017)Chen, Gao, Wei and Wang}]{Chen:2017AMC}
\bibinfo{author}{Chen\xfnm[ Y.Q.]}, \bibinfo{author}{Gao\xfnm[ Q.]},
  \bibinfo{author}{Wei\xfnm[ Y.H.]}, \bibinfo{author}{Wang\xfnm[ Y.]}.
\newblock \bibinfo{title}{Study on fractional order gradient methods}.
\newblock \bibinfo{journal}{Applied Mathematics \& Computation}
  \bibinfo{year}{2017};\bibinfo{volume}{314}:\bibinfo{pages}{310--321}.
\bibitem[{Chen et~al.(2018)Chen, Wei, Wang and Chen}]{Chen:2018ACC}
\bibinfo{author}{Chen\xfnm[ Y.Q.]}, \bibinfo{author}{Wei\xfnm[ Y.H.]},
  \bibinfo{author}{Wang\xfnm[ Y.]}, \bibinfo{author}{Chen\xfnm[ Y.Q.]}.
\newblock \bibinfo{title}{Fractional order gradient methods for a general class
  of convex functions}.
\newblock In: \bibinfo{booktitle}{2018 Annual American Control Conference
  (ACC)}. \bibinfo{address}{Milwaukee, USA}; \bibinfo{year}{2018}, p.
  \bibinfo{pages}{3763--3767}.
\bibitem[{Wang et~al.(2017)Wang, Wen, Gou, Ye and Chen}]{Wang:2017NN}
\bibinfo{author}{Wang\xfnm[ J.]}, \bibinfo{author}{Wen\xfnm[ Y.Q.]},
  \bibinfo{author}{Gou\xfnm[ Y.D.]}, \bibinfo{author}{Ye\xfnm[ Z.Y.]},
  \bibinfo{author}{Chen\xfnm[ H.]}.
\newblock \bibinfo{title}{Fractional-order gradient descent learning of bp
  neural networks with caputo derivative}.
\newblock \bibinfo{journal}{Neural Networks}
  \bibinfo{year}{2017};\bibinfo{volume}{89}:\bibinfo{pages}{19--30}.
\bibitem[{Bao et~al.(2018)Bao, Pu and Yi}]{Bao:2018CIN}
\bibinfo{author}{Bao\xfnm[ C.H.]}, \bibinfo{author}{Pu\xfnm[ Y.F.]},
  \bibinfo{author}{Yi\xfnm[ Z.]}.
\newblock \bibinfo{title}{Fractional-order deep backpropagation neural
  network}.
\newblock \bibinfo{journal}{Computational Intelligence \& Neuroscience}
  \bibinfo{year}{2018};\bibinfo{volume}{2018}:\bibinfo{pages}{1--10}.
\bibitem[{Rumelhart et~al.(1986)Rumelhart, Hinton, Williams
  et~al.}]{Rumelhart:1988Nature}
\bibinfo{author}{Rumelhart\xfnm[ D.E.]}, \bibinfo{author}{Hinton\xfnm[ G.E.]},
  \bibinfo{author}{Williams\xfnm[ R.J.]}, et~al.
\newblock \bibinfo{title}{Learning representations by back-propagating errors}.
\newblock \bibinfo{journal}{Nature}
  \bibinfo{year}{1986};\bibinfo{volume}{323}:\bibinfo{pages}{533--536}.

\end{thebibliography}







\end{document}